\documentclass[11pt]{article}

\renewcommand{\theequation}{\thesection.\arabic{equation}} 
\usepackage{amsmath}    
\usepackage{amsthm}     
\usepackage{amsfonts}   
\usepackage{amssymb}    
\usepackage{mathrsfs}   
\usepackage{stmaryrd}   
\usepackage{color}      
\usepackage{graphicx}   
\usepackage{float}      
\usepackage{placeins}   
\usepackage{enumerate}  
\usepackage{subeqnarray}
\usepackage{verbatim}   
\usepackage{cases}      
\usepackage{needspace}  
\usepackage{booktabs}   
\usepackage{caption}    
\usepackage{titlesec}   
\usepackage{appendix}   

\newtheorem{Theorem}{Theorem}[section]   
\newtheorem{Lemma}[Theorem]{Lemma}       
\newtheorem{Definition}[Theorem]{Definition} 
\newtheorem{Remark}[Theorem]{Remark}     

\newcommand{\beq}{\begin{equation}}
\newcommand{\enq}{\end{equation}}
\newcommand{\eeq}{\end{equation}}
\newcommand{\beqa}{\begin{eqnarray}}
\newcommand{\eeqa}{\end{eqnarray}}
\newcommand{\bqa}{\begin{eqnarray}}
\newcommand{\eqa}{\end{eqnarray}}

\newcommand{\br}{\begin{Remark}}
\newcommand{\er}{\end{Remark}}

\newcommand{\edm}{\end{displaymath}}

\allowdisplaybreaks[3]     
\usepackage{hyperref}
\usepackage{zref-user}

\begin{document}
\newpage
\pagenumbering{arabic} \setcounter{page}{1}

\begin{center}
    {\Large  Error Analysis of Tr-PINNs Algorithm for 2D Incompressible Navier-Stokes Equations with Non-Homogeneous Boundary Conditions 
    }\\[0.25in]
    Dongjie Liu\footnotemark[1] \quad
    Xuebo Li\footnotemark[2] \quad
    Rong Yang \footnotemark[1] 
    
    \vspace{0.1in}
    \footnotetext[1]{School of Mathematics, Statistics and Mechanics, Beijing University of Technology, 
    Ping Le Yuan 100, Chaoyang District, Beijing 100124, P.~R.~China. }
    \footnotetext[2]{School of Science, Chongqing University of Technology, Chongqing 400054, P.~R.~China.}
 \footnotetext { E-mail addresses: ldj2001x@163.com (D. Liu), xbl@cqut.edu.cn (X. Li), ysihan2010@163.com (R. Yang).}
\end{center}

\begin{abstract}
   Physics-informed neural networks (PINNs) have been widely applied to solve Navier-Stokes equations by enforcing outputs and gradients of deep models to satisfy target equations. However, conventional PINNs only constrain the boundary terms by means of the $L^2$-norm when addressing the equations with non-homogeneous boundary conditions. This single constraint strategy may  cause inaccurate boundary simulation, further resulting in the decline of prediction accuracy. To resolve this critical issue, this paper proposes an improved physics-informed neural network by correcting the error of the boundary value, which is called  Tr-PINNs.
Based on the results of nonhomogeneous Stokes problem, the algorithm error analysis of Tr-PINNs is established. The efficacy of the Tr-PINNs algorithm is demonstrated via numerical experiments, which further demonstrate that the Tr-PINNs algorithm achieves a remarkable improvement in computational accuracy.

\bigskip


\end{abstract}

\section{Introduction}

\textbf{PINNs applications in PDEs}. Partial differential equations (PDEs) are a core mathematical tool for describing complex physical processes. These processes are widely studied in fields including continuum mechanics, fluid dynamics and quantum physics. The accurate solution of these equations is directly related to the reliability of scientific analysis and the rationality of engineering design. Thus, developing efficient and robust numerical methods to obtain reliable high-accuracy approximate solutions of PDEs is one of the core objectives of scientific computing. With the research objects evolving toward cross-scale, strong nonlinearity and high-dimensionality, the challenges faced by traditional numerical methods for PDEs (e.g., finite element method, finite difference method) have become increasingly prominent, such as mesh dependence and exponential growth in computational cost. However, the rapid development of deep learning has provided an important opportunity for innovative breakthroughs in the field of PDEs solving.  The deep integration of deep learning with PDEs solving is driven to become a cutting-edge trend in the current computational physics field.

Against this backdrop, the physics-informed deep learning paradigm of ``integrating deep learning with physical priors" has emerged, and PINNs serve as the core solution framework of this paradigm. Raissi et al. (2019) first proposed a systematic methodological framework for physics-informed neural networks (PINNs). By transforming PDEs governing equations, initial conditions and boundary conditions into soft constraints incorporated into the loss function, the neural network can fit observed data and comply with underlying physical laws simultaneously in the training phase. This enables the unified solution of both forward and inverse PDEs problems within the mesh-free computational paradigm. The PINNs framework not only fully leverages the superior modeling capability of deep learning, but also effectively addresses the mesh dependence and high-dimensional bottlenecks inherent to traditional PDEs numerical methods, thus establishing a foundational academic basis for PINNs within computational physics \cite{RPK2019}.
Shin et al. (2020) pioneered in providing rigorous mathematical proofs for the applicability of PINNs to PDEs problems, from the standpoints of convergence and generalization. This work thereby lays a foundational theoretical groundwork for the application of PINNs in solving PDEs \cite{SD2020}. Mishra and Molinaro (2023) established the generalization error upper bound for PINNs when approximating the solutions to the forward problems of PDEs. They  derived that the generalization error estimates are dependent on the training error and the number of training samples \cite{MM2023}. De Ryck and Mishra (2024) established a crucial foundation for the numerical exploration of PINNs and associated models within physics-informed machine learning. Their comprehensive work further furnished substantial theoretical underpinnings for the development of this discipline \cite {DM2024}. For more improved literature on PINNs solving other types of PDEs, we refer
to \cite{BML2025, CMWYK2021, KLML2025, KP2025, LMK2021, RHG2022, WHKCB2022,WLWML2024,WYP2022}.\\

\noindent \textbf{PINNs applications in Navier-Stokes equations}. In this paper, we consider the following  boundary value problem of 2D incompressible Navier-Stokes equations 
\begin{eqnarray}\label{eq1.1}
\begin{aligned}
  & \left\{
  \begin{array}{ll}
  \frac{\partial u}{\partial t}- \nu \Delta u + u \cdot \nabla u  + \nabla p = f,  &x\in\Omega,\quad t>0,  \\
     \nabla \cdot u=0,   &x\in\Omega,\quad t>0,\\
    u |_{t=0}=u_0,   &x\in\Omega, \\
     u  =g,   &x\in\partial\Omega, ~~ t>0,\\
      \end{array}
  \right.
  \end{aligned}
 \end{eqnarray}   
in a bounded domain $\Omega$ with locally Lipschitzian or $C^2$ boundary $\partial\Omega$, where  $u=(u_{1},u_{2})$ is the fluid velocity field, $p$ is the  fluid   pressure.
  $\nu>0$ denote the viscosity.
  
Subsequent studies have continued to deepen around the theoretical rigor, method adaptability and scenario extensibility of PINNs.  At the theoretical level, Jin et al. (2021) proposed ``NSFnets" providing an adaptive PINNs solution for such equations, which is a PINN framework specifically designed for solving incompressible Navier-Stokes equations \cite{JCLK2021}. Biswas et al. (2022) proposed a PINN framework with an artificial regularization term for the 2D incompressible Navier-Stokes equations with homogeneous boundary conditions, and performed rigorous error and stability analyses via the Hodge decomposition \cite{BTU2022}. Xiao et al. (2023) further combined PINNs with Reynolds-averaged Navier-Stokes equations to model the Rayleigh-Taylor turbulent mixing process \cite{XZYY2023}.
In \cite{TRMA2024}, Thakur et al. (2024) proposed a method for assigning relative weights to the mean squared loss terms in the objective function used for training PINNs.   2D and 3D Navier-Stokes equations were employed as illustrative examples.  In \cite{DJM2024}, De Ryck et al. (2024) derived rigorous error estimates for the approximate solutions of PINNs targeting the Navier-Stokes equations. They also quantified the correlation between network structure and solution accuracy. In  \cite{BFNP2025}, Botarelli et al. (2025) achieved the solution of Navier-Stokes equations in complex fluid dynamics scenarios by using PINNs. They expanded their application scope in practical complex flow fields.  Guo et al. (2025) also systematically summarized the relevant progress of PINNs in solving complex PDEs (including Navier-Stokes equations) and engineering applications \cite{GZYG2025}.  Khademi and Dufour (2025) improved the PINNs structure by using trainable sine activation functions for approximating the solutions of Navier-Stokes equations \cite{KD2025}.\\

\noindent \textbf{Conventional PINNs algorithm}. In the architecture of the PINNs, the neural network is structured as follows: it comprises an input layer dedicated to accommodating coordinate variables $(x, y)$, which is followed by $n$ hidden layers each consisting of $m$ neurons, activation and culminates in an output layer that yields predictive values for the two variables $(u_h,p_h)$. When the PINNs is sufficiently trained and the network has enough capacity, $(u_h,p_h)$
 is intended to approximate these true weak solutions $(u,p)$ for the Navier-Stokes problem. We define the pointwise residuals
 \begin{eqnarray}\label{eq.1.4}
\begin{aligned} 
  & \left\{
  \begin{array}{ll}
  \mathcal{R}_{pde}=\frac{\partial u_h}{\partial t} - \nu \Delta u_h + u_h \cdot \nabla u_h  + \nabla p_h - f,  &x\in\Omega,\quad t>0,  \\
  \mathcal{R}_{div}=   \nabla \cdot u_h, &x\in\Omega,\quad t>0,  \\
  \mathcal{R}_{init}=  u_h|_{t=0} - u_0,  &x\in\Omega,\\
  \mathcal{R}_{bdy}= u_h -g,   &x\in\partial\Omega,~~ t>0,\\
      \end{array}
  \right.
  \end{aligned}
 \end{eqnarray}  
where \(u_h: \Omega \to \mathbb{R}^2\) and $p_h: \Omega \to \mathbb{R}$ are constructed by minimizing the generalization error.
Referring to \cite{DJM2024, GZYG2025,MM2023, RPK2019}, the definition of the generalization error of conventional PINNs is as follows
\begin{align}
  \widetilde{\varepsilon}_G(h) = 
  \|\mathcal{R}_{pde}\|^2_{L^2(0,T;L^2(\Omega))}+\|\mathcal{R}_{init}\|^2_{L^2(\Omega)} + \|\mathcal{R}_{div}\|^2_{L^2(0,T;L^2(\Omega))}+\|\mathcal{R}_{bdy}\|^2_{L^2(0,T;{L^2(\partial \Omega))}}.
\end{align}

However, this raises a critical issue about the error analysis of PINNs algorithm: if $(u,p)$ is the solution to the partial differential equation \eqref{eq1.1}, given a PINNs $(u_h,p_h)$ with small generalization error, does there exist a suitable norm $\|\cdot\|$ and a nonnegative function $\delta(\cdot)$  satisfying $\delta(\varepsilon)\rightarrow 0$ as $\varepsilon\rightarrow 0$,  such that 
 \begin{eqnarray}\label{error analysis}
\|u_h-u\| \leq \delta(\widetilde{\varepsilon}_G(h))?
\end{eqnarray}
Next, we give an affirmative answer by correcting the generalization error.\\

\noindent \textbf{The Tr-PINNs algorithm error analysis}. The aforementioned studies have advanced the deep integration of deep learning and PDEs solving across three dimensions: theoretical research, methodological innovation, and practical application. Consequently, these efforts have established PINNs as one of the core cutting-edge directions in contemporary computational physics. However, in sharp contrast to the widespread application of PINNs, it is found that there are few papers that rigorously demonstrate the error analysis of the newly proposed PINNs algorithms for solving the Dirichlet boundary of Navier-Stokes equations. In practical engineering scenarios, due to objective factors such as irregular geometric boundaries, discontinuous material properties and sudden load changes, the solutions of PDEs may hardly meet the requirement of high smoothness (e.g., $H^2$ and higher-order smoothness). In fact, the solutions to a large number of engineering problems only belong to $H^1$. Based on the unique weak solution to \eqref{eq1.1}, this paper proposes the Tr-PINNs algorithm, which is a theoretical innovation. More precisely, compared with the conventional PINNs algorithm by using  $L^2$-norm to measure boundary residuals, this paper employs the trace theorem in mathematics to define the boundary residual by using $H^{\frac{1}{2}}$-norm. There are the following two advantages to doing this correction: 
\begin{enumerate}
	\item[1)] The error analysis can be conducted based on the Tr-PINNs algorithm, which verifies the theoretical rigor of the proposed algorithm. Compared with \cite{DJM2024},  the weak solution is considered in this paper instead of classical solution.

	\item[2)] In Section 4,  three numerical examples show that the Tr-PINNs algorithm significantly improves the computational accuracy.
	
\end{enumerate}

 Based on the trace norm, we find a neural network $(u_h,p_h)$ named as Tr-PINNs algorithm, for which all residuals are simultaneously minimized by minimizing the strengthened generalization error
\begin{align}\label{minque}
  \varepsilon_G(h) = 
  \|\mathcal{R}_{pde}\|^2_{L^2(0,T;L^2(\Omega))}+\|\mathcal{R}_{init}\|^2_{L^2(\Omega)} + \|\mathcal{R}_{div}\|^2_{H^1(0,T;L^2(\Omega))}+\|\mathcal{R}_{bdy}\|^2_{H^1(0,T;{H^\frac{1}{2}(\partial \Omega))}}.
\end{align}
Therefore, the boundary error is effectively optimized via the proposed method.

Then  the error estimates result can be establishes as follows:
 \begin{Theorem}\label{erroranay}
  Let $(u, p)\in  V\times  {L}^{2}$ be the unique weak solution of equation \eqref{eq1.1}. Suppose  \((u_h, p_h)\) is constructed via Tr-PINNs as the sufficiently trained numerical approximations of $(u, p)$. Then for all $T>0$,  there exists a constant $C =C(\Omega,T ,\nu)$  such that
  $$
  \|u_h-u\|^2_{L^\infty(0,T;L^2(\Omega))} \leqslant C \varepsilon_G(h). 
  $$
\end{Theorem}
Therefore, the result means that the  solution becomes more accurate only when the generalization error is trained to be sufficiently small.\\

\noindent \textbf{The training error}. Since the quantity $\varepsilon_G(h)$  involves integrals, it
can not be directly minimized in practice. Thanks to the method justified by Lemma \ref{lemma6}, the integrals in \eqref{minque} are approximated by Monte Carlo Integration, resulting in
\begin{align}
  \varepsilon^{i_1}_T(h,S_i) &= \frac{|\Omega|  T}{N_i}\sum_{n=1}^{N_i}|\mathcal{R}_{pde}(t^i_n,x^i_n)|^2, \quad \varepsilon^{i_2}_T(h,S_i) = \frac{|\Omega| T}{N_i}\sum_{n=1}^{N_i}|\mathcal{R}_{div}(t^i_n,x^i_n)|^2, \notag\\
  \varepsilon^{init}_T(h,S_{init}) &= \frac{ |\Omega|  }{N_{init}}\sum_{n=1}^{N_{init}}|\mathcal{R}_{init}(0,x^{init}_{n})|^2,\quad   \varepsilon^{b_1}_T(h,S_{b_1}) = \frac{|\partial\Omega | T}{N_b}\sum_{n=1}^{N_b}|\mathcal{R}_{bdy}(t^b_n,x^b_n)|^2 ,\\
  \varepsilon^{b_2}_T(h,S_{b_2}) &= \frac{|\partial\Omega|^2  T}{N_b k}\sum_{i=1}^{N_b} \sum_{j \in \mathcal{N}_k(i)} \frac{|\mathcal{R}_{bdy}(t^b_i,x^b_i)-\mathcal{R}_{bdy}(t^b_i,x^b_j)|^2}{|x_i-x_j|^2},\notag
\end{align}
with quadrature points in space constituting data sets 
\begin{align}
S_i          &= \{(t^i_n,x^i_n)\}^{N_i}_{n=1}, \qquad \quad \quad \quad (t^i_n,x^i_n) \in [0,T] \times \Omega,\notag \\
S_{init}     &= \{(0,x^{init}_n)\}^{N_{init}}_{n=1}, \qquad \qquad \, \, x^{init}_n  \in \Omega, \\
S_{b_1}      &= \{(t^b_n,x^b_n)\}^{N_b}_{n=1}, \qquad \qquad \quad (t^b_n,x^b_n) \in [0,T] \times\partial\Omega, \notag\\
S_{b_2}      &= \{( t_i^b, x_i^b, x_j^b ) \}_{\substack{i=1 , j \in \mathcal{N}_k(i)}}^{N_b}, \quad (t^b_i,x^b_i,x^b_j) \in [0,T] \times\partial\Omega\times\partial\Omega, \notag
\end{align}
here \(j \in \mathcal{N}_k(i)\): denotes that sample $j$ belongs to the k-nearest neighbor set of sample $i$-here. Notice that the k-nearest neighbor set is defined to exclude the sample $i$ itself, which ensures that the denominator is non-zero.

According to generalization error defined by \eqref{minque},  the training error is denoted as
\begin{equation}
  \begin{split}
  \varepsilon_T(h) = & \varepsilon^{i_1}_T(h,S_i)+\varepsilon^{i_2}_T(h,S_i) +\varepsilon^{init}_T(h,S_{init})+\varepsilon^{b_1}_T(h,S_{b_1})+\varepsilon^{b_2}_T(h,S_{b_2}) \\
  &+\frac{\partial \varepsilon^{i_2}_T(h,S_i)}{\partial t}+\frac{\partial \varepsilon^{b_1}_T(h,S_{b_1})}{\partial t}  +\frac{\partial  \varepsilon^{b_2}_T(h,S_{b_2})}{\partial t}.
  \end{split}
\end{equation}
If the sample size is sufficiently large, this process has been proven by Lemma \ref{lemma6} that  the generalization error can be made sufficiently small once
the training error is reduced to an extremely small value. Therefore, the essence of PINNs lies in training the neural network parameters $h$ to minimize the training error. \\

\noindent \textbf{Organization of the paper}.  In Section 2,  several functional spaces and lemmas for subsequent use are presented. Specially, the wellposedness of the weak solution to \eqref{eq1.1} is established.
The result of error analysis of Tr-PINNs for system \eqref{eq1.1} is proved in Section 3. Three numerical
examples are presented in Section 4. The conclusions of this paper are presented in Section 5. Additionally, the proof of wellposedness for system \eqref{eq1.1} and some lemmas required for the preceding sections are  supplemented in Appendix A.

\section{Preliminary and main results}
 \setcounter{equation}{0}

We begin by introducing some spaces and one lemma, which are commonly used in the mathematical study of fluids. 
Let $\mathbb{L}^p(\Omega):=(L^p(\Omega))^2$, $\mathbb{H}^{1}(\Omega):=(H^1(\Omega))^2$ be the integral spaces and Sobolev spaces respectively, where $L^p(\Omega)$ and $H^1(\Omega)$ are  defined as usual.
Denoting 
$$\mathcal{V}:=\big\{u|~u\in \mathbb{H}^{1}(\Omega),\ \mbox{div} ~u=0 \big\},$$
$H$ is the closure of $\mathcal{V}$ in ${\mathbb{L}^2}(\Omega)$ topology, $\|\cdot\|$ and
$(\cdot,\cdot)$ denote the norm and inner product in $H$ respectively, where,  for $u, v\in {\mathbb{L}^2}(\Omega)$,
$$
(u,v)=\sum\limits_{\scriptstyle i=1}^2\int_{\Omega}u_i(x)v_i(x)dx,
$$
and $\|\cdot\|=\|\cdot\|_2$.
Let $V$ be the closure of the set $\mathcal{V}$ in $\mathbb{H}^{1}(\Omega)$
topology, $\|\cdot\|_{\mathbb{H}^{1}}$ and $((\cdot,\cdot))$ denote the norm and inner product in $V$ respectively.
Clearly, $V\hookrightarrow
H\equiv H'\hookrightarrow V'$, $H'$ and $V'$ are dual spaces of $H$ and $V$ respectively, where the injection is dense and
continuous. $\|\cdot\|_{*}$ and $\langle\cdot\rangle$ denote the norm in $V'$ and the dual product between $V$ and $V'$ respectively. 

For all $u, v, w\in \mathbb{H}^{1}(\Omega)$, the trilinear operator is defined by
$$
b(u,v,w)=((u \cdot \nabla )v,w)=\sum\limits_{\scriptstyle i,j=1}^2\int_\Omega u_i (D_i v_j) w_j dx
$$
where $D_i=\frac{\partial}{\partial x_i}$.
\begin{Lemma}\label{lemma2.1}
  The trilinear operators have the following relationship
  \begin{equation}
      \begin{gathered}
   b(u,v,w)=-b(u,w,v), \quad \forall u \in V\cap \mathbb{H}_0^{1}(\Omega),   v,w\in \mathbb{H}^{1}(\Omega),\\
 b(u,v,w)=-b(u,w,v), \quad \forall u \in  V,   v\in \mathbb{H}^{1}(\Omega), w\in \mathbb{H}_0^{1}(\Omega).
        \end{gathered}
  \end{equation}
   moreover, it is easy to verify that 
   \begin{eqnarray}
    b(u,v,v)=0, \quad \quad \forall u \in V\cap \mathbb{H}_0^{1}(\Omega),   v\in \mathbb{H}^{1}(\Omega).
   \end{eqnarray}
 \end{Lemma}
 
Next, the  wellposedness results of system \eqref{eq1.1} is stated in the following, which is proved  in the Appendix. 
\begin{Theorem}\label{Unique1.1}
  Let  $(f,g) \in L^2(0,T;V') \times H^1(0,T;\mathbb{H}^{\frac{1}{2}}(\partial\Omega))$ and $u_0 \in H$.
    Then there exists a unique weak solution $(u, p)$ to system \eqref{eq1.1} with
   \begin{eqnarray}\label{stability}
(u, p)\in  \big( L^2(0,T; V) \cap L^\infty(0,T;H) \big)\times L^2(0,T; {L}^{2}(\Omega)).
  \end{eqnarray}
  \end{Theorem}
\begin{Remark}
The nonhomogeneous boundary condition of $g\in H^1(0,T;\mathbb{H}^{\frac{1}{2}}(\partial\Omega))$ is not the best. Referring to \cite{R2007}, the regularity requirement of the  condition $g$ can be optimized by means of Helmholtz projection decomposition and semigroup theory.
\end{Remark}

\section{The error analysis of Tr-PINNs}
\setcounter{equation}{0} \ \ \ \
     Let $d=2$,   $\nu >0$ and $(u,p)\in \big( L^2(0,T;V) \cap L^\infty(0,T;H) \big)\times L^2(0,T;L^2(\Omega)) $ be the weak solution of Navier-Stokes equation \eqref{eq1.1}. If $(u_h, p_h)$ is a approximating solution constructed via the  Tr-PINNs with parameters vector $h$,     
     then the resulting $L^{2}$-error can be obtained as follows. In the next proof, the constants $C$ may be different, which just depend on $\Omega, T ,\nu$.
\par\vspace{8pt}
\noindent{\bf Proof of Theorem \ref{erroranay}.} Let $\hat{u}=u_{h}-u$ and $\hat{p}=p_{h}-p$ denote the difference between
the solution of the Navier-Stokes equations and the approximating solution. Combining the Navier-Stokes equations \eqref{eq1.1}
and pointwise residuals equations \eqref{eq.1.4},   a straightforward calculation shows that
\begin{equation}
\begin{cases}
\mathcal{R}_{pde}=\frac{\partial \hat{u}}{\partial t}-\nu\Delta \hat{u}+\hat{u}\cdot\nabla \hat{u}+u\cdot \nabla\hat{u}+\hat{u}\cdot\nabla u+\nabla \hat{p},&\forall x \in  \Omega \quad t>0,
\\
\mathcal{R}_{div}=\nabla\cdot\hat{u},\quad  &\forall x \in  \Omega \quad t>0,
\\
\mathcal{R}_{init}=\hat{u}(0), &\forall x \in  \Omega,
\\
\mathcal{R}_{bdy}=\hat{u}, &\forall x \in \partial \Omega ~~ t>0.
\end{cases}
\label{eq4.3}
\end{equation}

It follows from  Lemma  \ref{lemma3}   that there exists $w_{h}$ with
\begin{eqnarray}\label{3.2}
w_h - u=(u_h - u)-\widetilde{u}_h=\hat{u}-\widetilde{u}_h,
\end{eqnarray}
satisfying that
\begin{equation}
\begin{split} 
\nabla\cdot (w_h -u)&=\mathcal{R}_{div}\quad  \forall x \in  \Omega \quad t>0; \quad w_h - u=\mathcal{R}_{bdy} \quad \forall x \in \partial \Omega ~~ t>0,\\
&\text{and}\quad \|w_h -u\|_{H^1(\Omega)} \leq C (\| \mathcal{R}_{bdy}\|_{H^\frac{1}{2}(\partial\Omega)}+\|\mathcal{R}_{div}\|_{L^2(\Omega)}),
\label{stokesest}
\end{split}
\end{equation}
then it can be obtained that $(u_h-u)|_{\partial\Omega}= (w_h-u)|_{\partial\Omega}$ and 
\begin{eqnarray}\label{defre}
\nabla\cdot \widetilde{u}_h=0 \quad \forall x \in  \Omega \quad t>0; \quad \quad  \widetilde{u}_h|_{\partial\Omega}=0.
\end{eqnarray}

The following equation is obtained by taking the inner product of the first equation of (\ref{eq4.3}) with $\widetilde{u}_h $
\begin{equation}
\begin{split}\label{3.4}
\langle \widetilde{u}_h ,\mathcal{R}_{pde}\rangle=
&\langle \widetilde{u}_h ,\frac{\partial \hat{u}}{\partial t}\rangle -\nu\langle \widetilde{u}_h ,\Delta\hat{u}\rangle +b(\hat{u},\hat{u},\widetilde{u}_h )  \\
&+b(u,\hat{u},\widetilde{u}_h )+b(\hat{u},u,\widetilde{u}_h )+\langle \widetilde{u}_h ,\nabla \hat{p}\rangle.
\end{split}
\end{equation}
By substituting \eqref{3.2} into \eqref{3.4} and simplifying appropriately, the following is obtained
\begin{equation}
\begin{split}\label{3.5}
\langle\widetilde{u}_h ,\mathcal{R}_{pde}\rangle=
& \langle\widetilde{u}_h ,\frac{\partial \widetilde{u}_h }{\partial t}\rangle +\langle \widetilde{u}_h ,\frac{\partial (w_h-u)}{\partial t}\rangle -\nu \langle \widetilde{u}_h ,\Delta \widetilde{u}_h \rangle -\nu\langle \widetilde{u}_h ,\Delta (w_h-u)\rangle \\
& +b(\widetilde{u}_h ,\widetilde{u}_h ,\widetilde{u}_h )+b((w_h-u),(w_h-u),\widetilde{u}_h )+b((w_h-u),\widetilde{u}_h ,\widetilde{u}_h )\\
& +b(\widetilde{u}_h ,(w_h-u),\widetilde{u}_h )+b(\widetilde{u}_h ,u,\widetilde{u}_h )+b(u,\widetilde{u}_h ,\widetilde{u}_h )\\
& +b(u,(w_h-u),\widetilde{u}_h )+b((w_h-u),u,\widetilde{u}_h )+\langle\widetilde{u}_h ,\nabla \hat{p}\rangle.
\end{split}
\end{equation} 
In what follows, each term in the equation above is processed term by term. Firstly, by \eqref{stokesest}, the time derivative terms can be dealt by
\begin{equation}\label{3.6}
\begin{split}
\langle\widetilde{u}_h ,\frac{\partial \widetilde{u}_h }{\partial t}\rangle
& =\frac{1}{2}\frac{d}{d t}\|\widetilde{u}_h  \|^2_{L^2(\Omega)},\\
\langle\widetilde{u}_h ,\frac{\partial (w_h-u)}{\partial t}\rangle
&  \leq  \epsilon \| \widetilde{u}_h \|^2_{L^2(\Omega)} + \frac{1}{4\epsilon} \|\frac{\partial( w_h-u)}{\partial t}\|^2_{L^2(\Omega)}\\
& \leq  \epsilon \| \widetilde{u}_h \|^2_{L^2(\Omega)} + \frac{C}{4\epsilon} (\|\frac{\partial \mathcal{R}_{bdy}}{\partial t}\|^2_{H^\frac{1}{2}(\partial \Omega)}+\|\frac{\partial \mathcal{R}_{div}} {\partial t}\|^2_{L^2(\Omega)}).
\end{split}
\end{equation}
According to Green's formula, the viscous diffusion terms can be estimated by
\begin{equation}
\begin{split}
\nu\langle\widetilde{u}_h ,\Delta \widetilde{u}_h \rangle
&=-\nu\|\nabla \widetilde{u}_h \|^2_{L^2(\Omega)},\\
|\nu\langle\widetilde{u}_h ,\Delta (w_h-u)\rangle|
&= |-\nu\langle\nabla \widetilde{u}_h ,\nabla (w_h-u)\rangle| \\
& \leq\nu \|\nabla \widetilde{u}_h \|_{L^2(\Omega)}\|\nabla (w_h-u)\|_{L^2(\Omega)}  \\
& \leq \nu\epsilon \|\nabla \widetilde{u}_h \|^2_{L^2(\Omega)}+\frac{1}{4\epsilon}\|\nabla (w_h-u)\|^2_{L^2(\Omega)} \\
& \leq \nu\epsilon \|\nabla \widetilde{u}_h \|^2_{L^2(\Omega)}+\frac{C}{4\epsilon}(\| \mathcal{R}_{bdy}\|^2_{H^\frac{1}{2}(\partial\Omega)}+\|\mathcal{R}_{div}\|^2_{L^2(\Omega)}) .
\end{split}
\end{equation} 
By Lemma \ref{lemma2.1}, it can be obtained that
\begin{equation}
|b(\widetilde{u}_h ,\widetilde{u}_h ,\widetilde{u}_h )|=0,\quad  |b(u,\widetilde{u}_h ,\widetilde{u}_h )|=0.
\end{equation} 
By means of Hölder's inequality, Embedding theorem and  Gagliardo-Nirenberg inequality, the following terms are bounded as
\begin{subequations}
\begin{flalign}
 |b((w_h-u),(w_h-u),\widetilde{u}_h )|
&\leq \| \widetilde{u}_h \|_{L^4 (\Omega)}\|\nabla (w_h-u)\|_{L^2(\Omega)} \|w_h-u\|_{L^4(\Omega)} \notag\\
&\leq \epsilon \| \nabla \widetilde{u}_h \|^2_{L^2 (\Omega)}+\frac{C}{4\epsilon}\sup\limits_{t\in[0,T]}\|w_h-u\|^4_{H^1(\Omega)} \notag\\
&\leq   \epsilon \| \nabla \widetilde{u}_h \|^2_{L^2 (\Omega)}+\frac{C}{4\epsilon} \|w_h-u\|^4_{H^1(0,T,H^1(\Omega))} \\
& \leq \epsilon \| \nabla \widetilde{u}_h \|^2_{L^2 (\Omega)}+\frac{C}{4\epsilon} (\|\mathcal{R}_{bdy}\|^4_{H^1(0,T,H^{\frac{1}{2}}(\Omega))}+\|\mathcal{R}_{div}\|^4_{H^1(0,T,L^2(\Omega))}), \notag\\
  |b((w_h-u),\widetilde{u}_h ,\widetilde{u}_h )|
& \leq \|w_h-u\|_{L^4(\Omega)}\|\nabla \widetilde{u}_h\|_{L^2(\Omega)}\|\widetilde{u}_h\|_{L^4(\Omega)} \notag\\
&\leq  C\|w_h-u\|_{H^1(\Omega)}\|\nabla \widetilde{u}_h \|^{\frac{3}{2}}_{L^2 (\Omega)} \| \widetilde{u}_h \|^{\frac{1}{2}}_{L^2 (\Omega)} \notag\\
&\leq \epsilon \|\nabla \widetilde{u}_h \|^2_{L^2 (\Omega)} + \frac{C}{4\epsilon} (\sup\limits_{t\in[0,T]}\|w_h-u\|^4_{H^1(\Omega)}) \| \widetilde{u}_h \|^2_{L^2 (\Omega)}\\
&\leq \epsilon \|\nabla \widetilde{u}_h \|^2_{L^2 (\Omega)} + \frac{C}{4\epsilon} (\|\mathcal{R}_{bdy}\|^4_{H^1(0,T,H^{\frac{1}{2}}(\Omega))}+\|\mathcal{R}_{div}\|^4_{H^1(0,T,L^2(\Omega))}) \| \widetilde{u}_h \|^2_{L^2 (\Omega)} , \notag\\
|b(\widetilde{u}_h ,(w_h-u),\widetilde{u}_h )|
&\leq \|\nabla (w_h-u)\|_{L^2(\Omega)} \| \widetilde{u}_h \|^2_{L^4 (\Omega)}\notag \\
&\leq C\| w_h-u\|_{H^1(\Omega)}\|\nabla \widetilde{u}_h \|_{L^2 (\Omega)} \| \widetilde{u}_h \|_{L^2 (\Omega)}\\
&\leq \epsilon  \|\nabla \widetilde{u}_h \|^2_{L^2 (\Omega)}  + \frac{C}{4\epsilon}(\| \mathcal{R}_{bdy}\|^2_{H^\frac{1}{2}(\partial\Omega)}+\|\mathcal{R}_{div}\|^2_{L^2(\Omega)})\| \widetilde{u}_h \|^2_{L^2 (\Omega)}, \notag \\
|b(\widetilde{u}_h ,u,\widetilde{u}_h )|
&\leq \| \widetilde{u}_h \|^2_{L^4 (\Omega)}\|\nabla u\|_{L^2(\Omega)}\notag\\
&\leq C\| \widetilde{u}_h \|_{L^2 (\Omega)}\|\nabla \widetilde{u}_h \|_{L^2 (\Omega)}\|\nabla u\|_{L^2(\Omega)}  \\
&\leq \epsilon \|\nabla \widetilde{u}_h \|^2_{L^2 (\Omega)} + \frac{C}{4 \epsilon} \| u\|^2_{H^1(\Omega)}\| \widetilde{u}_h \|^2_{L^2 (\Omega)} , \notag\\
| b(u,(w_h-u),\widetilde{u}_h )|
& \leq \| u\|_{L^4 (\Omega)}\|\nabla (w_h-u)\|_{L^2(\Omega)} \|\widetilde{u}_h \|_{L^4(\Omega)} \\
& \leq C (\| \mathcal{R}_{bdy}\|_{H^\frac{1}{2}(\partial\Omega)}+\|\mathcal{R}_{div}\|_{L^2(\Omega)}) \|u\|^{\frac{1}{2}}_{H^1(\Omega)}\| \widetilde{u}_h \|^{\frac{1}{2}}_{L^2 (\Omega)} \|\nabla \widetilde{u}_h \|^{\frac{1}{2}}_{L^2 (\Omega)} \notag \\
& \leq \epsilon  \|u\|_{H^1(\Omega)} \| \widetilde{u}_h \|_{L^2 (\Omega)} \|\nabla \widetilde{u}_h \|_{L^2 (\Omega)} +\frac{C}{4\epsilon} (\| \mathcal{R}_{bdy}\|^2_{H^\frac{1}{2}(\partial\Omega)}+\|\mathcal{R}_{div}\|^2_{L^2(\Omega)})\notag\\
& \leq \epsilon  \|\nabla \widetilde{u}_h \|^2_{L^2 (\Omega)} + \frac{1}{4} \|u\|^2_{H^1(\Omega)} \| \widetilde{u}_h \|^2_{L^2 (\Omega)}   +\frac{C}{4\epsilon}(\| \mathcal{R}_{bdy}\|^2_{H^\frac{1}{2}(\partial\Omega)}+\|\mathcal{R}_{div}\|^2_{L^2(\Omega)}). \notag
\end{flalign}
\end{subequations}
here the embedding theorem $(w_h - u) \in H^1(0,T;H^1(\Omega))\hookrightarrow C([0,T];H^1(\Omega))$ is used.

With respect to $|b((w_h-u),u,\widetilde{u}_h )|= |\langle \nabla \cdot (w_h - u) , u \cdot \widetilde{u}_h \rangle  + b((w_h-u),\widetilde{u}_h,u)|$, it is processed by means of the solution space
\begin{flalign}
  |\langle \nabla \cdot (w_h - u) , u \cdot \widetilde{u}_h \rangle|
&\leq \|\mathcal{R}_{div}\|_{L^2(\Omega)}\|u\|_{L^4(\Omega)} \|\widetilde{u}_h\|_{L^4(\Omega)} \notag\\
&\leq C \|\mathcal{R}_{div}\|_{L^2(\Omega)}  \|u\|^{\frac{1}{2}}_{L^2(\Omega)}  \|u\|^{\frac{1}{2}}_{H^1(\Omega)}\| \widetilde{u}_h \|^{\frac{1}{2}}_{L^2 (\Omega)} \|\nabla \widetilde{u}_h \|^{\frac{1}{2}}_{L^2 (\Omega)}\notag\\
& \leq C \|\mathcal{R}_{div}\|_{L^2(\Omega)}  \|u\|^{\frac{1}{2}}_{H^1(\Omega)}\| \widetilde{u}_h \|^{\frac{1}{2}}_{L^2 (\Omega)} \|\nabla \widetilde{u}_h \|^{\frac{1}{2}}_{L^2 (\Omega)} \notag\\
& \leq \epsilon  \|u\|_{H^1(\Omega)} \| \widetilde{u}_h \|_{L^2 (\Omega)} \|\nabla \widetilde{u}_h \|_{L^2 (\Omega)} +\frac{C}{4\epsilon} \|\mathcal{R}_{div}\|^2_{L^2(\Omega)} \\
& \leq \epsilon  \|\nabla \widetilde{u}_h \|^2_{L^2 (\Omega)} + \frac{1}{4} \|u\|^2_{H^1(\Omega)} \| \widetilde{u}_h \|^2_{L^2 (\Omega)}   +\frac{C}{4\epsilon} \|\mathcal{R}_{div}\|^2_{L^2(\Omega)}  ,\notag\\
|b((w_h-u),\widetilde{u}_h,u)| 
&\leq C\|w_h - u\|_{L^4(\Omega)}\| \nabla \widetilde{u}_h \|_{L^2 (\Omega)}\| u\|_{L^4(\Omega)}    \notag\\
& \leq C \| \nabla \widetilde{u}_h \|_{L^2 (\Omega)} \|w_h - u\|^{\frac{1}{2}}_{L^2(\Omega)}\|w_h - u\|^{\frac{1}{2}}_{H^1(\Omega)} \|u\|^{\frac{1}{2}}_{L^2(\Omega)}\|u\|^{\frac{1}{2}}_{H^1(\Omega)}   \notag\\
& \leq \epsilon  \| \nabla \widetilde{u}_h \|^2_{L^2 (\Omega)} +\frac{C}{4\epsilon} \|w_h - u\|_{L^2(\Omega)}\|w_h - u\|_{H^1(\Omega)}\|u\|_{H^1(\Omega)}  \notag\\
& \leq \epsilon \| \nabla \widetilde{u}_h \|^2_{L^2 (\Omega)}  + \frac{1 }{4}\|u\|^2_{H^1(\Omega)} \sup\limits_{t\in[0,T]}\|w_h - u\|^2_{L^2(\Omega)}\\
&\quad + \frac{C }{4\epsilon}(\| \mathcal{R}_{bdy}\|^2_{H^\frac{1}{2}(\partial\Omega)}+\|\mathcal{R}_{div}\|^2_{L^2(\Omega)})\notag \\
& \leq \epsilon \| \nabla \widetilde{u}_h \|^2_{L^2 (\Omega)}  + \frac{C }{4\epsilon}(\| \mathcal{R}_{bdy}\|^2_{H^\frac{1}{2}(\partial\Omega)}+\|\mathcal{R}_{div}\|^2_{L^2(\Omega)})\notag \\
&\quad +\frac{C }{4}\|u\|^2_{H^1(\Omega)}(\|\mathcal{R}_{bdy}\|^2_{H^1(0,T;H^\frac{1}{2}(\partial \Omega))}+\|\mathcal{R}_{div}\|^2_{H^1(0,T;L^2( \Omega))}),\notag \\
|b((w_h-u),u,\widetilde{u}_h )| 
&\leq |\langle \nabla \cdot (w_h - u) , u \cdot \widetilde{u}_h \rangle| + | b((w_h-u),\widetilde{u}_h,u)| \notag\\
& \leq 2\epsilon \| \nabla \widetilde{u}_h \|^2_{L^2 (\Omega)}  + \frac{C }{4\epsilon}(\| \mathcal{R}_{bdy}\|^2_{H^\frac{1}{2}(\partial\Omega)}+\|\mathcal{R}_{div}\|^2_{L^2(\Omega)} )\\
&\quad + \frac{C }{4}\|u\|^2_{H^1(\Omega)}(\|\widetilde{u}_h\|^2_{L^2(\Omega)} +\|\mathcal{R}_{bdy}\|^2_{H^1(0,T;H^\frac{1}{2}(\partial \Omega))}+\|\mathcal{R}_{div}\|^2_{H^1(0,T;L^2( \Omega))}) \notag .
\end{flalign}
Recalling \eqref{defre}, the coupling pressure term is estimated by means of the Divergence theorem
\begin{equation}
\begin{split}
|\langle\widetilde{u}_h ,\nabla\hat{p} \rangle|
&=|\langle \nabla\cdot \widetilde{u}_h ,\hat{p} \rangle| = 0.
\end{split}
\end{equation} 
For the last remaining term, we can directly bound it by Hölder's inequality
\begin{equation}\label{3.15}
\begin{split}
|\langle \widetilde{u}_h ,\mathcal{R}_{pde} \rangle|
&\leq \| \widetilde{u}_h \|_{L^2 (\Omega)}\|\mathcal{R}_{pde} \|_{L^2(\Omega)}\\
&\leq \epsilon\|\widetilde{u}_h \|^2_{L^2 (\Omega)}+\frac{1}{4\epsilon}\| \mathcal{R}_{pde}\|^2_{L^2(\Omega)}.
\end{split}
\end{equation}
Now by combining  \eqref{3.5}-\eqref{3.15} and  choosing an appropriate $\epsilon$, we obtain that
\begin{flalign}\label{3.16}
      &\frac{d}{d t}\|\widetilde{u}_h \|^2_{L^2(\Omega)}+\nu\|\nabla \widetilde{u}_h \|^2_{L^2(\Omega)}  \leq C \Big(1+ \|\mathcal{R}_{bdy}\|^4_{H^1(0,T,H^{\frac{1}{2}}(\Omega))}+\|\mathcal{R}_{div}\|^4_{H^1(0,T,L^2(\Omega))} + \| \mathcal{R}_{bdy}\|^2_{H^\frac{1}{2}(\partial\Omega)} \notag \\
      & +\|\mathcal{R}_{div}\|^2_{L^2(\Omega)} + \|u\|^2_{H^1(\Omega)}\Big) \|\widetilde{u}_h\|^2_{L^2(\Omega)}  + C\|u\|^2_{H^1(\Omega)}(\|\mathcal{R}_{bdy}\|^2_{H^1(0,T;H^\frac{1}{2}(\partial \Omega))}+\|\mathcal{R}_{div}\|^2_{H^1(0,T;L^2( \Omega))})\notag \\
      & + C(\| \mathcal{R}_{bdy}\|^2_{H^\frac{1}{2} (\partial\Omega)}+ \|\frac{\partial \mathcal{R}_{bdy}}{\partial t}\|^2_{H^\frac{1}{2}(\partial \Omega)}  +\| \mathcal{R}_{div}\|^2_{L^2 (\Omega)}+ \|\frac{\partial \mathcal{R}_{div}}{\partial t}\|^2_{L^2( \Omega)}+ \| \mathcal{R}_{pde}\|^2_{L^2(\Omega)}) \notag \\
      & + C(\|\mathcal{R}_{bdy}\|^4_{H^1(0,T,H^{\frac{1}{2}}(\Omega))} + \|\mathcal{R}_{div}\|^4_{H^1(0,T,L^2(\Omega))}) .
\end{flalign}
 Since the residual term is an infinitesimal value if the model converges, without loss of generality, assume that
\begin{equation}\notag
  \|\mathcal{R}_{bdy}\|^4_{H^1(0,T,H^{\frac{1}{2}}(\Omega))}\leq \|\mathcal{R}_{bdy}\|^2_{H^1(0,T;H^\frac{1}{2}(\partial \Omega))}<1, \, \|\mathcal{R}_{div}\|^4_{H^1(0,T,L^2(\Omega))} \leq \|\mathcal{R}_{div}\|^2_{H^1(0,T;L^2( \Omega))}<1.
\end{equation}
If $\nu >0$, by Gronwall's inequality and $u\in L^2(0,T;V) \cap L^\infty(0,T;H) $, it is obtained that 
\begin{equation}
\begin{split}\label{3.17}
  \|\widetilde{u}_h (t)\|^2_{L^2(\Omega)} &\leq e^{C(1+t)} \big( \|\widetilde{u}_h (0)\|^2_{L^2(\Omega)} +\| \mathcal{R}_{pde}\|^2_{L^2(0,T;L^2(\Omega))}+\| \mathcal{R}_{div}\|^2_{H^1(0,T;L^2(\Omega))} \\
  &\quad + \| \mathcal{R}_{bdy}\|^2_{H^1(0,T;H^\frac{1}{2}(\partial \Omega))} \big), \quad \forall t \in [0,T].
\end{split}
\end{equation}
For the initial velocity,  one has
\begin{equation}
\begin{split}\label{3.18}
  \|\widetilde{u}_h (0)\|^2_{L^2(\Omega)}=
    &\|\hat{u}(0) -(w_h-u)(0)\|^2_{L^2(\Omega)}\\
    &\leq\|\hat{u}(0)\|^2_{L^2(\Omega)}+\sup\limits_{t\in[0,T]}\|(w_h-u)(t)\|^2_{L^2(\Omega)}\\
    &\leq \|\mathcal{R}_{init}\|^2_{L^2(\Omega)} + C( \|\mathcal{R}_{div}\|^2_{H^1(0,T;L^2(\Omega))}+\|\mathcal{R}_{bdy}\|^2_{H^1(0,T;{H^\frac{1}{2}(\partial \Omega))}}).
\end{split}
\end{equation}
Substituting \eqref{3.18} into \eqref{3.17}, we obtain
\begin{align}
\|\widetilde{u}_h (t)\|^2_{L^\infty(0,T;L^2(\Omega))}
&\leq C (\|\mathcal{R}_{pde}\|^2_{L^2(0,T;L^2(\Omega))}+\|\mathcal{R}_{init}\|^2_{L^2(\Omega)} \notag\\
&\quad + \|\mathcal{R}_{div}\|^2_{H^1(0,T;L^2(\Omega))}+\|\mathcal{R}_{bdy}\|^2_{H^1(0,T;{H^\frac{1}{2}(\partial \Omega))}}).
\end{align}
Finally, it is derived that
\begin{align}
\| u_h - u \|^2_{L^\infty(0,T;L^2(\Omega))}
& \leq \| \widetilde{u}_h\|^2_{L^\infty (0,T;L^2(\Omega))} + \sup\limits_{t\in[0,T]}\| w_h - u \|^2_{L^2(\Omega)} \notag \\
&\leq C (\|\mathcal{R}_{pde}\|^2_{L^2(0,T;L^2(\Omega))}+\|\mathcal{R}_{init}\|^2_{L^2(\Omega)} \notag\\
&\quad + \|\mathcal{R}_{div}\|^2_{H^1(0,T;L^2(\Omega))}+\|\mathcal{R}_{bdy}\|^2_{H^1(0,T;{H^\frac{1}{2}(\partial \Omega))}}),
\end{align}
which finishes the proof.
\hfill $\square$

\section{Numerical example}
\setcounter{equation}{0} 
Serving as a bridge between basic and applied research in fluid mechanics, the two-dimensional unsteady Navier-Stokes flow manifests its academic significance across three core dimensions: theoretical validation, methodological innovation and interdisciplinary empowerment. It has become an indispensable research paradigm within the discipline.

As a simplified testbed for investigating the nonlinear and unsteady features inherent to the Navier-Stokes equations, the two-dimensional Navier-Stokes flow mitigates the computational burden through its two-dimensional constraint, thereby establishing itself as a pivotal benchmark for assessing the stability, accuracy and convergence properties of numerical methods. 
Furthermore, considering that analytical solutions for such flows are rarely available, we construct a validation set consisting of three distinct test cases to evaluate the accuracy improvement of the Tr-PINNs algorithm. Specifically, two of these cases are widely used Taylor-Green vortex variants commonly employed in PINNs research (see e.g. \cite{DJM2024,LMK2021}), with the third case being independently selected for additional validation.\begin{Definition}\label{5.1}
Let \(R > 0\) be a bounded range for parameters, \(L \in \mathbb{N}\) denote the number of network layers,  \(l_0, l_1, \dots, l_L \in \mathbb{N}\) denote the widths of each layer (where \(l_0\) is the input dimension and \(l_L\) is the output dimension) and \(\sigma: \mathbb{R} \to \mathbb{R}\) be a twice-differentiable activation function (e.g., Tanh, Swish, SiLU, etc.).
For each layer index \(k \in \{1, 2, \dots, L\}\), define the local parameter space for the k-th layer as:
\begin{eqnarray}
\Theta_k = \left\{ \left( W_k, b_k \right) \ \bigg| \ W_k \in \mathbb{R}^{l_k \times l_{k-1}}, \ b_k \in \mathbb{R}^{l_k}, \ \text{and} \ \forall \, i,j, \ |(W_k)_{ij}| \leq R, \ \forall \, i, \ |(b_k)_i| \leq R \right\},
\end{eqnarray}
here \(W_k \in \mathbb{R}^{l_k \times l_{k-1}}\) is the weight matrix of the k-th layer, mapping from dimension \(l_{k-1}\) to \(l_k\),\\
\(b_k \in \mathbb{R}^{l_k}\) is the bias vector of the k-th layer.\\
The global parameter space \(\Theta\) is then the Cartesian product of these layer-wise local spaces:
\begin{eqnarray}
\Theta = \Theta_1 \times \Theta_2 \times \dots \times \Theta_L.
\end{eqnarray}
For any \(h \in \Theta\), define the linear transformation \(A_k^h: \mathbb{R}^{l_{k-1}} \to \mathbb{R}^{l_k}\) for the k-th layer as:
\begin{eqnarray}
A^{h}_k(z)=W_k z +b_k, \, \ \forall z \in \mathbb{R}^{l_{k-1}}.
\end{eqnarray}
Then, define the  transformation with activation$(\sigma)$ \(F_k^h: \mathbb{R}^{l_{k-1}} \to \mathbb{R}^{l_k}\) for the k-th layer as:
\begin{eqnarray}
F_k^h(z) = 
\begin{cases} 
W_k z + b_k & \text{if } k = L, \\
\sigma(W_k z + b_k) & \text{if } 1 \leq k < L .
\end{cases}
\end{eqnarray}
We define the neural network \(u_h: \mathbb{R}^{l_0} \to \mathbb{R}^{l_L}\)  as a composite function of the transformations across all layers:
\begin{eqnarray}
u_h(z) = \left( F_L^h \circ F_{L-1}^h \circ \dots \circ F_1^h \right)(z), \quad \forall z \in \mathbb{R}^{l_0}.
\end{eqnarray}
In the problem of approximating the Navier-Stokes equations, we take the input dimension \(l_0 = d\) and \(z = x\), \(u_h\) is used to approximate physical quantities (e.g., velocity, temperature, pressure) of the flow-thermal field and is referred to as the neural network realized by parameter \(h\).
\end{Definition}

\subsection{Network architecture and infrastructure}

Identical network architecture is guaranteed across all the cases presented below, where the methods adopted comprise Tr-PINNs, PINNs and CFD.

In this study, Tr-PINNs are compared with OpenFOAM simulations and conventional PINNs. For comparison fairness, a 64×64 grid is adopted in OpenFOAM, and the sampling configuration of the PINNs model is set to match the same grid scale as the CFD simulation, with a reasonable number of training iterations performed.

According to Definition \ref{5.1}, we parameterize $u_1(x,y,t,\nu)$, $u_2(x,y,t,\nu)$ and the pressure function $p(x,y,t,\nu)$ via a $6\times256$ fully connected neural network (6 hidden layers with 256 neurons per layer), and adopt the infinitely differentiable SiLU activation function for all network layers. Specifically, the input layer comprises 4 neurons corresponding to the input variables $(x, y, t,\nu)$; the hidden layers consist of a total of $6\times256$ neurons with the aforementioned SiLU activation; the output layer has 3 neurons, which represent the predicted values $u_{1h}(x,y,t,\nu)$, $u_{2h}(x,y,t,\nu)$ and $P_h(x,y,t,\nu)$, respectively.
Furthermore, we employ the Adam optimizer with an initial learning rate of 0.001 to update the network parameters $h$, and integrate this optimizer with the CosineAnnealingLR scheduler to dynamically adjust the learning rate throughout the training process.

The numerical simulations were performed on a workstation with the following hardware and software specifications:
\begin{itemize}
\item \textbf{CPU}: AMD Ryzen 9 8945HX with Radeon Graphics
\item \textbf{GPU}: NVIDIA GeForce RTX 5060 Laptop GPU
\item \textbf{Software}: Python 3.13.3, OpenFOAM v2312
\end{itemize}

\subsection{Case 1: variant 1 of the Taylor-Green vortex}
For the first validation case, we consider the Navier-Stokes equations \eqref{eq1.1} posed on the spatial domain $[0,\pi] \times [0,\pi]$, where the time horizon is $t\in [0,1]$ and the viscosity coefficient is set to $\nu \in [0.001,1]$. It can be inferred from $Re = \frac{\pi}{\nu}$ that $Re \in [3.14,3140]$. Let $f=0$, given the corresponding initial and boundary conditions as follows:
\begin{equation}
\begin{cases}
\begin{aligned}
u_1(0,y,t,\nu) &= -e^{-2 \nu t}\sin y, & u_2(0,y,t,\nu) &= 0,
\\
u_1(\pi,y,t,\nu) &= e^{-2 \nu t}\sin y, & u_2(\pi,y,t,\nu) &= 0,
\\
u_1(x,0,t,\nu) &= 0, & u_2(x,0,t,\nu) &= e^{-2 \nu t}\sin x,
\\
u_1(x,\pi,t,\nu) &= 0, & u_2(x,\pi,t,\nu) &= -e^{-2 \nu t}\sin x,
\\
u_1(x,y,0,\nu) &= -\cos x\sin y, & u_2(x,y,0,\nu) &= \sin x\cos y,
\end{aligned}
\end{cases}
\end{equation}
then the analytical solutions for the velocity and pressure of this test case are known, namely,
\begin{equation}
\begin{cases}
\begin{aligned}
u^*_1(x,y,t,\nu) &= -e^{-2 \nu t}\cos x \sin y,
\\ 
u^*_2(x,y,t,\nu) &= e^{-2 \nu t}\sin x \cos y,
\\
P^*(x,y,t,\nu) &= -\frac{1}{4}e^{-4 \nu t}(\cos 2x + \cos 2y).
\end{aligned}
\end{cases}
\end{equation}

Under strictly identical training environments, the comparison of $L^2$ relative errors of $u_1$ and $u_2$ among three algorithms across different Reynolds number (Re) scenarios at the operating condition of $t = 0.5$ is presented in Fig. \ref{fig4.1}. It is clearly observed that the aforementioned three algorithms exhibit extraordinary capability in simulating target solutions. Meanwhile, the $L^2$ relative errors of Tr-PINNs are found to be significantly lower than those of conventional PINNs across all Re numbers and significantly superior to the error levels of high-precision CFD simulations. Specifically, Tr-PINNs achieve an average error reduction of 54.58\% for $u_1$ and 49.82\% for $u_2$ compared with conventional PINNs; when compared with high-precision CFD simulations, the average error reduction further reaches 84.33\% for $u_1$ and 79.26\% for $u_2$.

Consistent with the preceding $L^2$ relative error analysis, the root mean square error (RMSE) of $u_1$ and $u_2$ for the three algorithms is further assessed under identical training environments across diverse Re scenarios at $t = 0.5$ in Fig. \ref{fig4.2}. While conventional PINNs demonstrate reasonable accuracy in capturing the target flow fields, the RMSE values of Tr-PINNs are consistently lower than those of conventional PINNs across all tested Re cases, and even surpass high-precision CFD simulations in terms of prediction precision. Specifically, for $u_1$, Tr-PINNs reduces the maximum RMSE by up to 66.88\% relative to conventional PINNs and by nearly 89.19\% compared with CFD; for $u_2$, the corresponding maximum reductions are approximately 69.59\% and 85.66\%, respectively. These RMSE results not only align with the $L^2$ error findings but also further corroborate that Tr-PINNs delivers superior numerical accuracy against both conventional PINNs and CFD, reinforcing its reliability and precision for simulating diverse operating conditions.

To better show its advantages, we select the visualization results at the operating condition of $t=0.5$, $\text{Re}=3141$ ($\nu=0.001$) to elaborate on the advantages of Tr-PINNs in detail with reference to Fig. \ref{fig4.3}.

As observed from panels (a)(b), (e)(f) and (i)(j) of Fig. \ref{fig4.3}, Tr-PINNs, PINNs and CFD methods exhibit excellent numerical fitting capabilities. For a fair comparison, we adopt the maximum absolute errors of $u_1$ and $u_2$ predicted by PINNs as the baseline threshold, and mark the regions where the absolute errors of Tr-PINNs and CFD exceed this threshold in yellow.

By analyzing the point-wise absolute error contours in (c)(d), (g)(h) and (k)(l) of Fig. \ref{fig4.3}, we find that the CFD method shows significant limitations in boundary simulation, with a singular point of maximum absolute error reaching 0.012453; in contrast, Tr-PINNs maintains high-precision simulation performance across the entire computational domain.

As a classical benchmark method in fluid mechanics, CFD serves as a reliable reference for validating model performance. To this end, we further plot the point-wise error contours between Tr-PINNs predictions and CFD results ((o)(p) of Fig. \ref{fig4.3}). To eliminate the interference of the singular point in CFD, we restrict the error range to $[-0.001, 0.001]$ and mark points outside this range in yellow. The results clearly demonstrate that Tr-PINNs is nearly identical to CFD, which strongly verifies the reliability of the Tr-PINNs method.

In summary, although the training time of Tr-PINNs is approximately 94 minutes, that of the PINN is 43 minutes, and the average training time of CFD is 102 seconds. Once the model is trained, the inference time of Tr-PINNs is merely on the millisecond scale. We acknowledge the indispensable role of CFD in fluid mechanics; nevertheless, the inference capability of Tr-PINNs will be fully demonstrated under multi-operating conditions, which greatly facilitates our subsequent research and workflow.
Moreover, through these aforementioned quantifiable error reductions, we unequivocally demonstrate that Tr-PINNs achieve substantially higher accuracy than PINNs and CFD. More importantly, such consistent error reduction across a range of $\text{Re}$ confirms that Tr-PINNs not only mitigate the inherent precision limitations of conventional PINNs but also outperform the accuracy benchmark of high-precision CFD. This endows them with robust and reliable numerical simulation capabilities for complex flow problems---an essential advance for engineering applications that demand high computational prediction fidelity.
\begin{figure}[htbp]
  \centering
  \renewcommand{\thefigure}{4.1} 
  \includegraphics[width=0.85\textwidth, keepaspectratio]{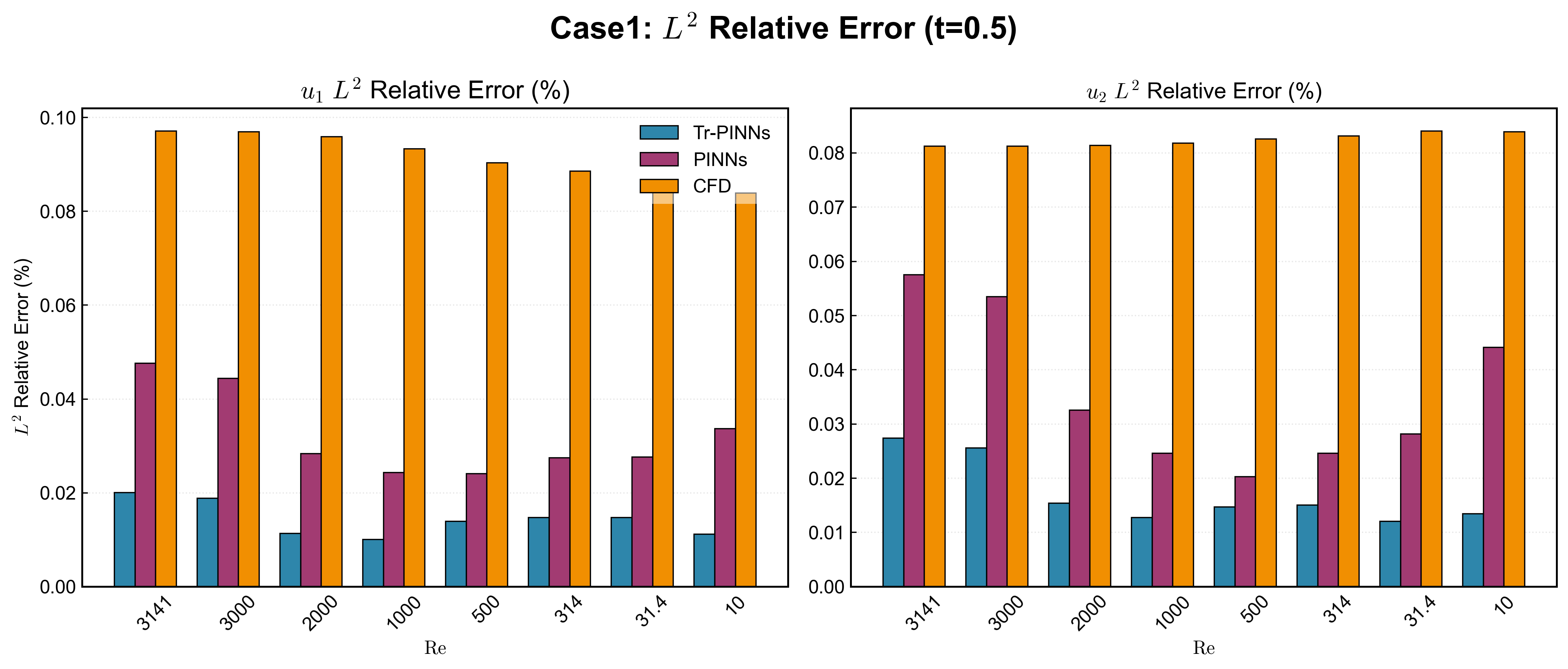}
  \caption{The first case: $L^2$ Relative Error of $u_1$ and $u_2$ under different working conditions at $t=0.5$.}
  \label{fig4.1}
\end{figure}

\begin{figure}[htbp]
  \centering
  \renewcommand{\thefigure}{4.2} 
  \includegraphics[width=0.85\textwidth, keepaspectratio]{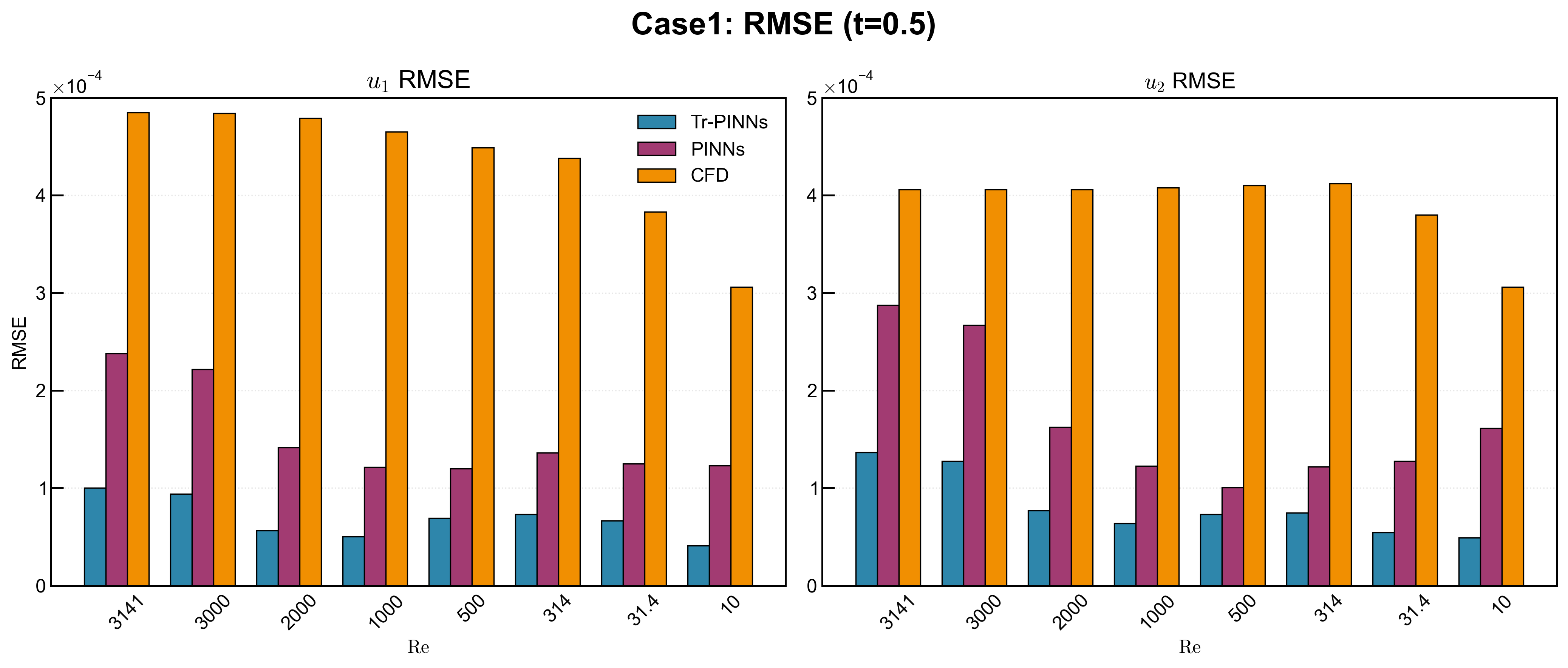}
  \caption{The first case: Root Mean Square Error of $u_1$ and $u_2$ under different working conditions at $t=0.5$.}
  \label{fig4.2}
\end{figure}

\begin{figure}[H]
  \centering
  \renewcommand{\thefigure}{4.3} 
  \includegraphics[width=1\textwidth, keepaspectratio]{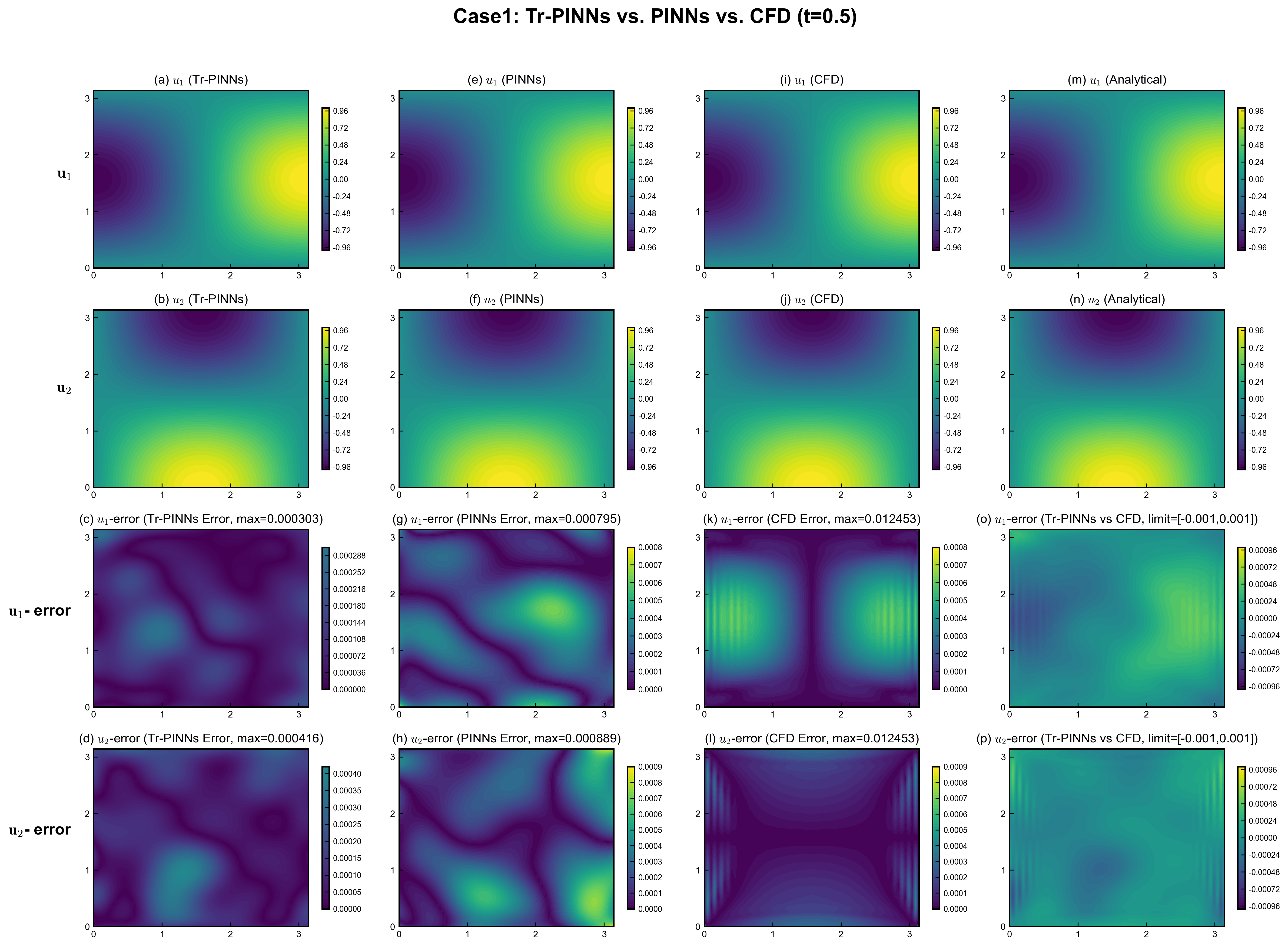}
  \caption{The first case: (a)(b) and (c)(d) show the predicted solutions and point-wise absolute errors of the Tr-PINNs algorithm; (e)(f) and (g)(h) show the predicted solutions and point-wise absolute errors of the PINNs algorithm; (i)(j) and (k)(l) show the numerical simulation results and point-wise absolute errors of OpenFOAM; (m)(n) denote the analytical solutions; (o)(p) show the point-wise errors between the Tr-PINNs predictions and OpenFOAM results.}
  \label{fig4.3}
\end{figure}

\subsection{Case 2: variant 2 of the Taylor-Green vortex}

For the second validation case, we consider the Navier-Stokes equations \eqref{eq1.1} posed on the spatial domain $[0,\pi] \times [0,\pi]$, where the time horizon is $t\in [0,1]$ and the viscosity coefficient is set to $\nu \in [0.001,1]$. It can be inferred from $Re = \frac{\pi}{\nu}$ that $Re \in [3.14,3140]$. Let $f=0$, given the corresponding initial and boundary conditions as follows:
\begin{equation}
\begin{cases}
\begin{aligned}
u_1(0,y,t,\nu) &= 0, & u_1(\pi,y,t,\nu) &= 0,
\\
u_2(0,y,t,\nu) &= e^{-2\nu t} cosy , & u_2(\pi,y,t,\nu) &= -e^{-2\nu t}cosy ,
\\
u_1(x,0,t,\nu) &= 0, & u_1(x,\pi,t,\nu) &= 0,
\\
u_2(x,0,t,\nu) &= e^{-2\nu t} cosx , & u_2(x,\pi,t,\nu) &= -e^{-2\nu t} cosx ,
\\
u_1(x,y,0,\nu) &= sinx siny , & u_2(x,y,0,\nu) &= cosx cosy,
\end{aligned}
\end{cases}
\end{equation}
then the analytical solution is presented by
\begin{equation}
\begin{cases}
\begin{aligned}
u^*_1(x,y,t,\nu) &= e^{-2\nu t} sinx siny,
\\ 
u^*_2(x,y,t,\nu) &= e^{-2\nu t} cosx cosy,
\\
P^*(x,y,t,\nu) &= \frac{1}{4}e^{-4\nu t}(cos2x -cos2y).
\end{aligned}
\end{cases}
\end{equation}

For the second test case in this paper, the employed neural network architecture, training configurations and CFD settings are identical to those in the first test case.

Under strictly identical training environments, the comparison of $L^2$ relative errors of $u_1$ and $u_2$ among three algorithms across different Re scenarios at the operating condition of $t = 0.5$ is presented in Fig. \ref{fig4.4}. The $L^2$ relative errors of Tr-PINNs are found to be significantly lower than those of conventional PINNs across all Re numbers and distinctly superior to the error levels associated with high-precision CFD simulations. Specifically, average error reductions of 31.99\% for $u_1$ and 36.91\% for $u_2$ are achieved by Tr-PINNs relative to conventional PINNs, and relative to high-precision CFD simulations, further average error reductions of 57.91\% for $u_1$ and 46.58\% for $u_2$ are obtained.

In line with the aforementioned analysis of $L^2$ relative errors, the RMSE of $u_1$ and $u_2$ for the three algorithms is further evaluated under identical training conditions across a range of Re scenarios at $t = 0.5$, as depicted in Fig. \ref{fig4.5}. The RMSE values derived from Tr-PINNs are persistently lower than those of conventional PINNs across all tested Re cases, and the predictive accuracy of Tr-PINNs even outperforms that of high-precision CFD simulations. Specifically, for $u_1$, Tr-PINNs achieves the maximal RMSE reduction of as high as 71.53\% relative to conventional PINNs and approximately 71.75\% in comparison with CFD simulations; for $u_2$, the maximal RMSE reduction relative to conventional PINNs is around 51.21\%, while that relative to CFD simulations reaches 64.01\%.

To better show its advantages, we select the visualization results at the operating condition of $t=0.5$, $\text{Re}=3141$ ($\nu=0.001$) to elaborate on the advantages of Tr-PINNs in detail with reference to Fig. \ref{fig4.6}.

As observed from panels (a)(b), (e)(f) and (i)(j) of Fig. \ref{fig4.6}, Tr-PINNs, PINNs and CFD alike possess exceptional ability to fit numerical data with high accuracy. For a fair comparison, we adopt the maximum absolute errors of $u_1$ and $u_2$ predicted by PINNs as the baseline threshold, and mark the regions where the absolute errors of Tr-PINNs and CFD exceed this threshold in yellow.

By analyzing the point-wise absolute error contours in (c)(d), (g)(h), and (k)(l) of Fig. \ref{fig4.6}, It can be clearly observed that under this operating condition, the maximum absolute error of $u_1$ for Tr-PINNs is reduced by 35.6\% relative to conventional PINNs and by approximately 43.60\% relative to CFD simulations; for $u_2$, the maximum absolute error of Tr-PINNs is decreased by 46.57\% compared with conventional PINNs and by around 40.36\% compared with CFD simulations. Meanwhile, the comparative analysis of the figures demonstrates the superior performance of Tr-PINNs in boundary fitting.

From the comparison between (o)  and (p) of Fig. \ref{fig4.6}, it is evident that the yellow regions almost entirely coincide with the zones of large errors in the CFD simulations. In contrast, the discrepancy between Tr-PINNs and CFD is close to zero in most regions, which further corroborates the superior fitting capability of Tr-PINNs in high-gradient and high-error regions.

To summarize, consistent with the aforementioned accuracy analyses, Tr-PINNs exhibit distinct superiority over both conventional PINNs and high-precision CFD simulations in terms of computational accuracy. Admittedly, Tr-PINNs require a longer training time. However, once well-trained, Tr-PINNs only consume millisecond-scale time for inference. While CFD remains a fundamental and indispensable tool in fluid dynamics research, the remarkable inference efficiency of Tr-PINNs greatly optimizes the efficiency of our subsequent research processes.
\begin{figure}[htbp]
  \centering
  \renewcommand{\thefigure}{4.4} 
  \includegraphics[width=0.85\textwidth, keepaspectratio]{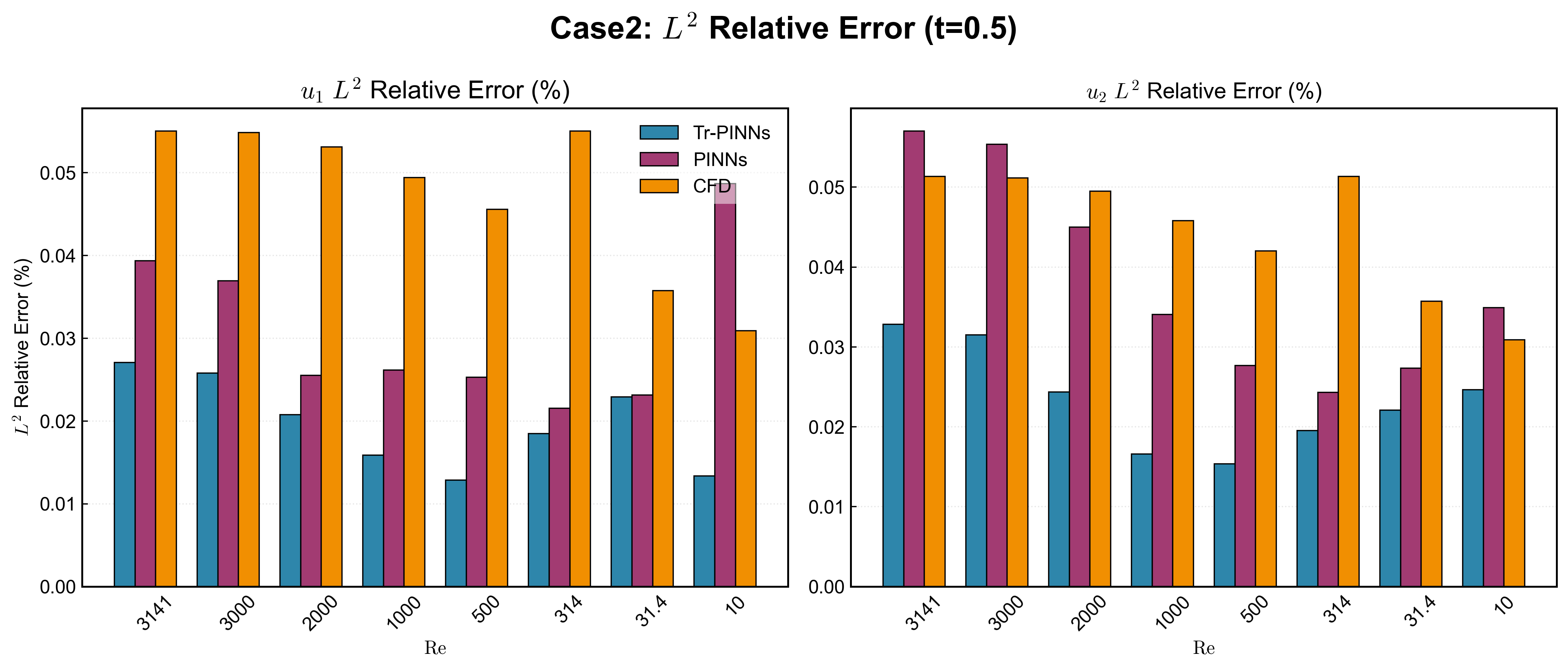}
  \caption{The second case: $L^2$ Relative Error of $u_1$ and $u_2$ under different working conditions at $t=0.5$.}
  \label{fig4.4}
\end{figure}

\begin{figure}[htbp]
  \centering
  \renewcommand{\thefigure}{4.5} 
  \includegraphics[width=0.85\textwidth, keepaspectratio]{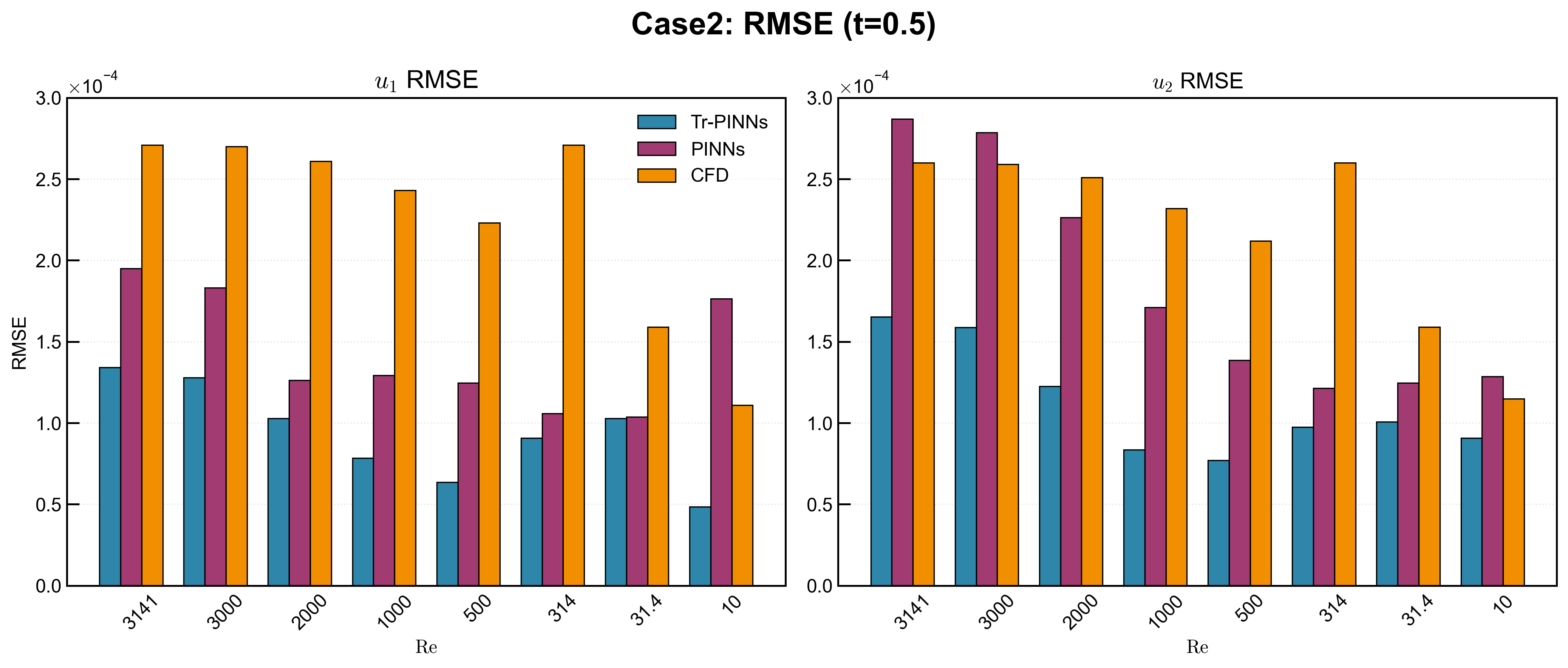}
  \caption{The second case: Root Mean Square Error of $u_1$ and $u_2$ under different working conditions at $t=0.5$.}
  \label{fig4.5}
\end{figure}

\begin{figure}[H]
  \centering
  \renewcommand{\thefigure}{4.6} 
  \includegraphics[width=1\textwidth, keepaspectratio]{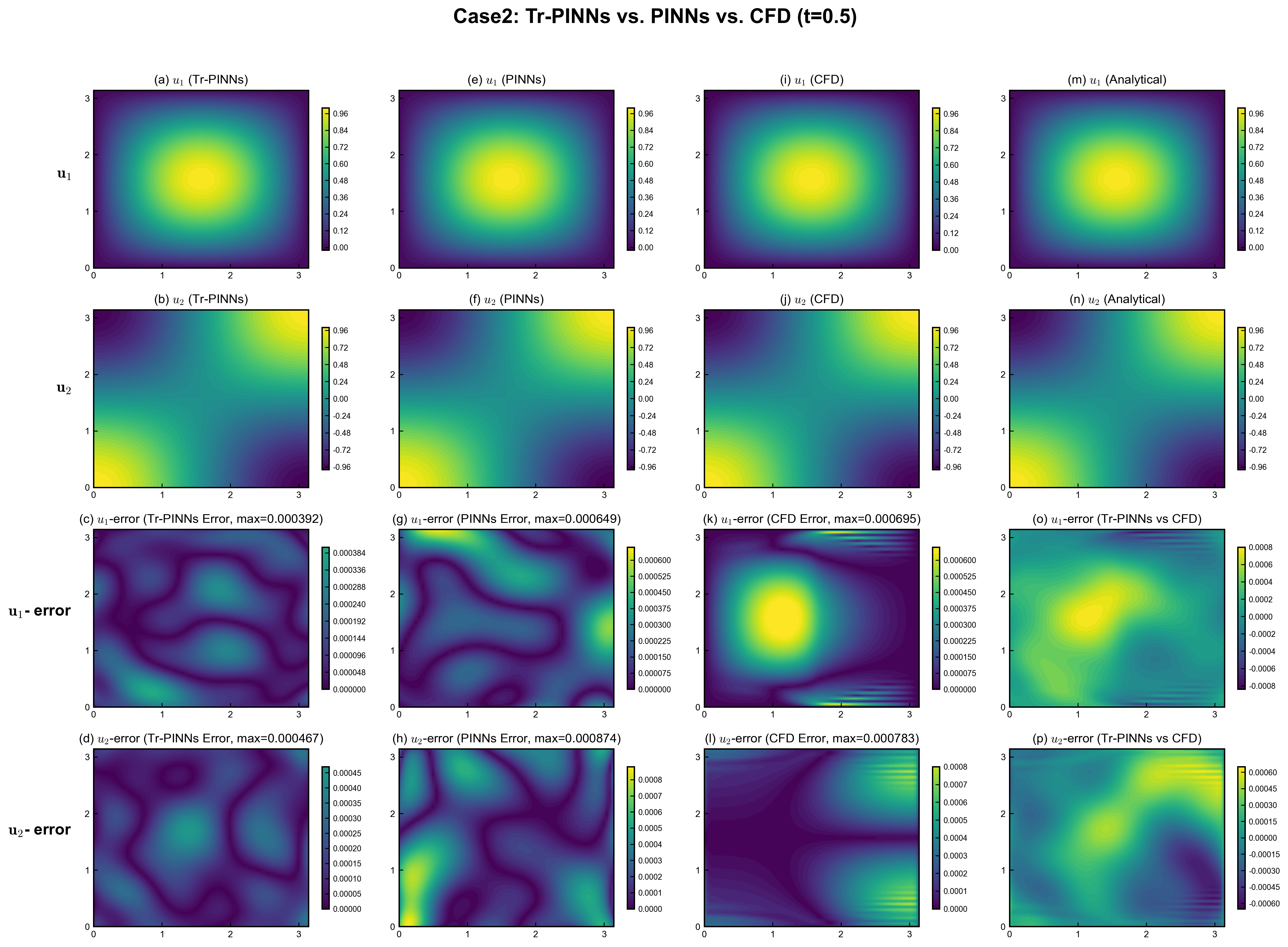}
  \caption{The second case: (a)(b) and (c)(d) show the predicted solutions and point-wise absolute errors of the Tr-PINNs algorithm; (e)(f) and (g)(h) show the predicted solutions and point-wise absolute errors of the PINNs algorithm; (i)(j) and (k)(l) show the numerical simulation results and point-wise absolute errors of OpenFOAM; (m)(n) denote the analytical solutions; (o)(p) show the point-wise errors between the Tr-PINNs predictions and OpenFOAM results.}
  \label{fig4.6}
\end{figure}

\subsection{Case 3: supplementary test case}
We further consider a numerical example to verify the improved accuracy of our Tr-PINNs algorithm. For the third validation case, we consider the Navier-Stokes equations \eqref{eq1.1} posed on the spatial domain $[0,\pi] \times [0,\pi]$, where the time horizon is $t\in [0,1]$ and the viscosity coefficient is set to $\nu \in [0.001,1]$. It can be inferred from $Re = \frac{\pi}{\nu}$ that $Re \in [3.14,3140]$. Let $f=0$, given the corresponding initial and boundary conditions as follows:
\begin{equation}
\begin{cases}
\begin{aligned}
u_1(0,y,t,\nu) &= e^{-5\nu t}\sin2y, & u_1(\pi,y,t,\nu) &= -e^{-5\nu t}\sin2y,
\\
u_2(0,y,t,\nu) &= -\frac{1}{2}e^{-5\nu t}\sin2y, & u_2(\pi,y,t,\nu) &= \frac{1}{2}e^{-5\nu t}\sin2y,
\\
u_1(x,0,t,\nu) &= e^{-5\nu t}\sin x, & u_1(x,\pi,t,\nu) &= e^{-5\nu t}\sin x,
\\
u_2(x,0,t,\nu) &= -\frac{1}{2}e^{-5\nu t}\sin x, & u_2(x,\pi,t,\nu) &= -\frac{1}{2}e^{-5\nu t}\sin x,
\\
u_1(x,y,0,\nu) &= \sin(x+2y), & u_2(x,y,0,\nu) &= -\frac{1}{2}\sin(x+2y),
\end{aligned}
\end{cases}
\end{equation}
then the analytical solution is presented
\begin{equation}
\begin{cases}
\begin{aligned}
u^*_1(x,y,t,\nu) &= e^{-5\nu t}\sin(x+2y),
\\ 
u^*_2(x,y,t,\nu) &= -\frac{1}{2}e^{-5\nu t}\sin(x+2y),
\\
P^*(x,y,t,\nu) &= C.
\end{aligned}
\end{cases}
\end{equation}
where $C$ is an arbitrary constant.

For the third test case in this paper, the employed neural network architecture, training configurations and CFD settings are identical to those in the first test case.

To simulate atmospheric pressure conditions, we prescribe $P^*(x,y,t,\nu) = 0$. In strictly consistent training setups, Fig. \ref{fig4.7} presents the comparison of $L^2$ relative errors for $u_1$ and $u_2$ among three algorithms under different Re number conditions at $t = 0.5$. Tr-PINNs are shown to possess significantly lower $L^2$ relative errors than conventional PINNs throughout all Re regimes, and perform drastically better than the error benchmarks of high-fidelity CFD simulations. Concretely, relative to conventional PINNs, Tr-PINNs realize average error reductions of 41.89\% for $u_1$ and 39.99\% for $u_2$; relative to high-fidelity CFD simulations, further average error reductions of 92.38\% for $u_1$ and 89.35\% for $u_2$ are accomplished.

Building upon the aforementioned analysis of $L^2$ relative errors, we further quantify RMSE of $u_1$ and $u_2$ for the three algorithms under uniformly controlled training conditions, across a broad range of Re cases at $t = 0.5$ (see Fig. \ref{fig4.8}). Notably, Tr-PINNs yield RMSE values that are consistently diminished relative to those of conventional PINNs across all tested Re regimes, and the predictive precision of Tr-PINNs even exceeds that of high-fidelity CFD simulations. Quantitatively, for $u_1$, Tr-PINNs attains a maximum RMSE reduction of 65.66\% over conventional PINNs and approximately 94.66\% relative to CFD simulations; for $u_2$, the maximum RMSE reduction against conventional PINNs is 65.64\%, whereas the reduction relative to CFD simulations reaches 89.99\%.

To better show its advantages, we select the visualization results at the operating condition of $t=0.5$, $\text{Re}=3141$ ($\nu=0.001$) to elaborate on the advantages of Tr-PINNs in detail with reference to Fig. \ref{fig4.9}.

As observed from panels (a)(b), (e)(f) and (i)(j) of Fig. \ref{fig4.9}, Tr-PINNs, PINNs and CFD methods exhibit excellent numerical fitting capabilities. For a fair comparison, we adopt the maximum absolute errors of $u_1$ and $u_2$ predicted by PINNs as the baseline threshold, and mark the regions where the absolute errors of Tr-PINNs and CFD exceed this threshold in yellow.

By analyzing the point-wise absolute error contours in (c)(d), (g)(h) and (k)(l) of Fig. \ref{fig4.9}. It can be clearly observed that the conventional PINNs method yields a relatively large residual at the peak point. Moreover, most data points appearing yellow in the CFD contour plots are characterized by absolute errors considerably larger than the peak absolute error of the PINNs method. In contrast, the overall contour of the Tr-PINNs method is presented in dark blue, with a slightly larger absolute error only observed in local regions. 

The point‑wise error contours between the predictions generated by Tr‑PINNs and the reference solutions obtained from CFD are further illustrated in panels (o) and (p) of Fig. \ref{fig4.9}. To suppress the disturbance caused by singular points in the CFD dataset, the error range is restricted to $[-0.001,0.001]$, and points beyond this interval are marked in yellow. 
It can be clearly observed that the errors are negligibly small in regions where both Tr-PINNs and CFD converge, while yellow regions persist in areas where CFD exhibits numerical inaccuracies. This sufficiently validates that Tr-PINNs is superior to conventional CFD methods in both prediction accuracy and numerical consistency.

In conclusion, the training time of Tr-PINNs is approximately 94 minutes, which is longer than the 44 minutes required for PINNs and also longer than the 104 seconds for CFD. Despite the marginal increase in training overhead, the millisecond-level inference efficiency of Tr-PINNs drastically reduces the computational cost for real-time prediction and engineering deployment, thereby offsetting the longer training duration.  
Owing to the inherent numerical errors and singular artifacts in conventional CFD, its reference benchmark displays significant deviations and non-physical fluctuations. In contrast to these distorted CFD solutions, Tr-PINNs yields an ultra-low, spatially homogeneous error distribution with no non-physical artifacts. This result clearly validates the prominent advantages of Tr-PINNs in prediction accuracy, numerical stability and computational practicality, showing its superiority over both traditional CFD and PINNs.
\begin{figure}[htbp]
  \centering
  \renewcommand{\thefigure}{4.7} 
  \includegraphics[width=0.85\textwidth, keepaspectratio]{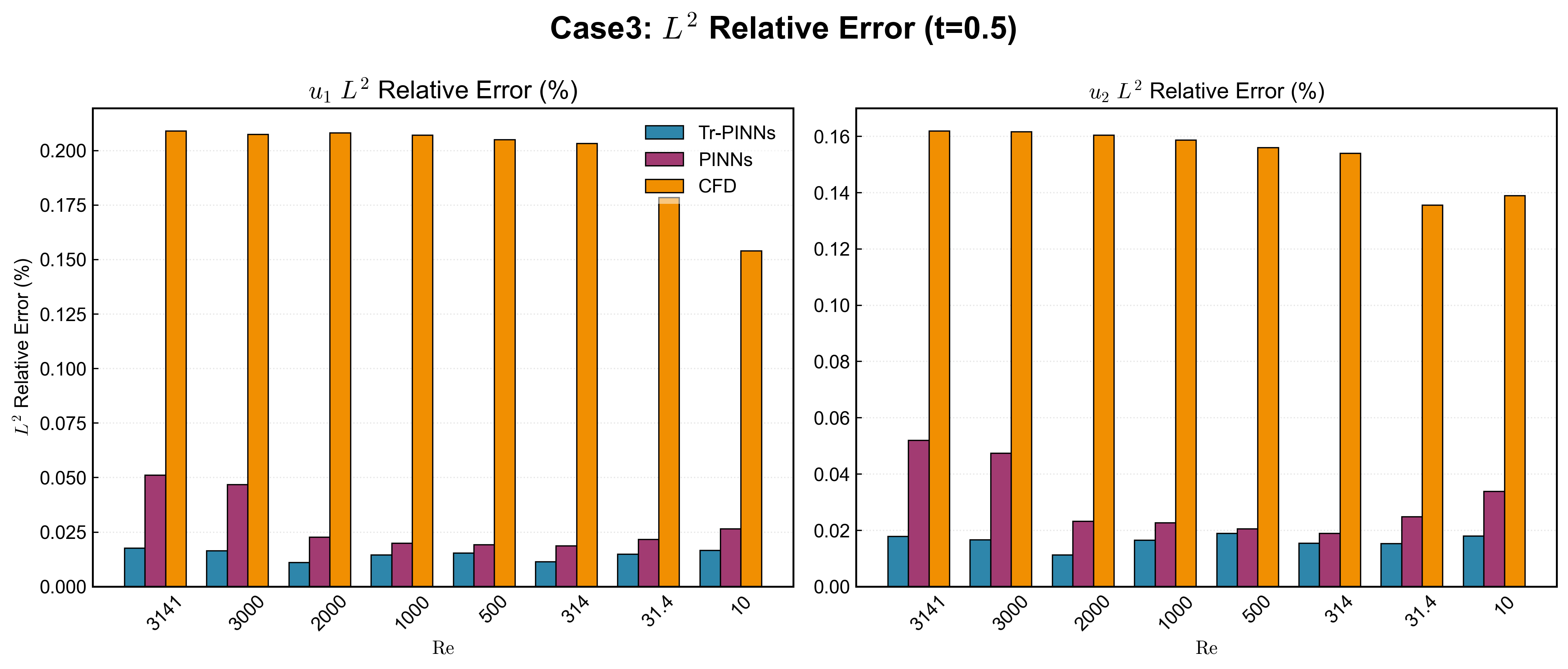}
  \caption{The third case: $L^2$ Relative Error of $u_1$ and $u_2$ under different working conditions at $t=0.5$.}
  \label{fig4.7}
\end{figure}

\begin{figure}[htbp]
  \centering
  \renewcommand{\thefigure}{4.8} 
  \includegraphics[width=0.85\textwidth, keepaspectratio]{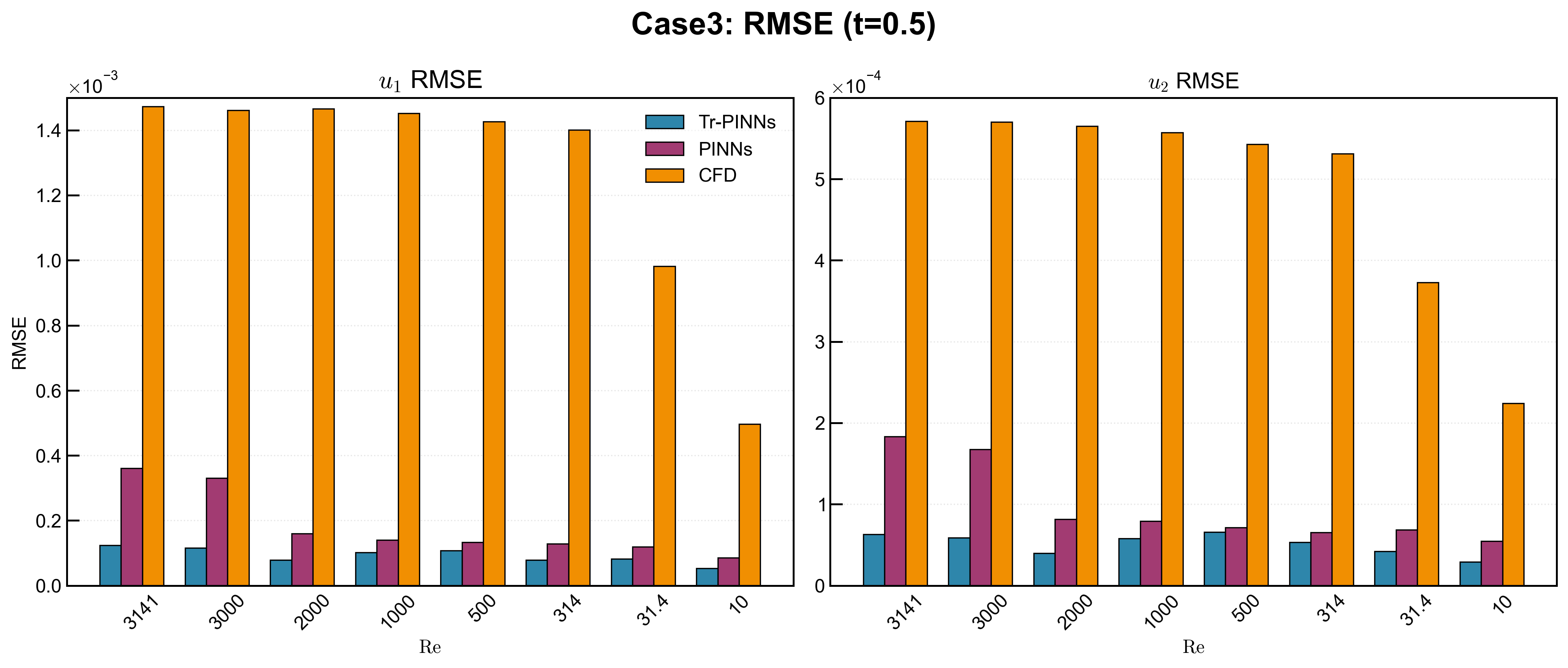}
  \caption{The third case: Root Mean Square Error of $u_1$ and $u_2$ under different working conditions at $t=0.5$.}
  \label{fig4.8}
\end{figure}

\begin{figure}[H]
  \centering
  \renewcommand{\thefigure}{4.9} 
  \includegraphics[width=1\textwidth, keepaspectratio]{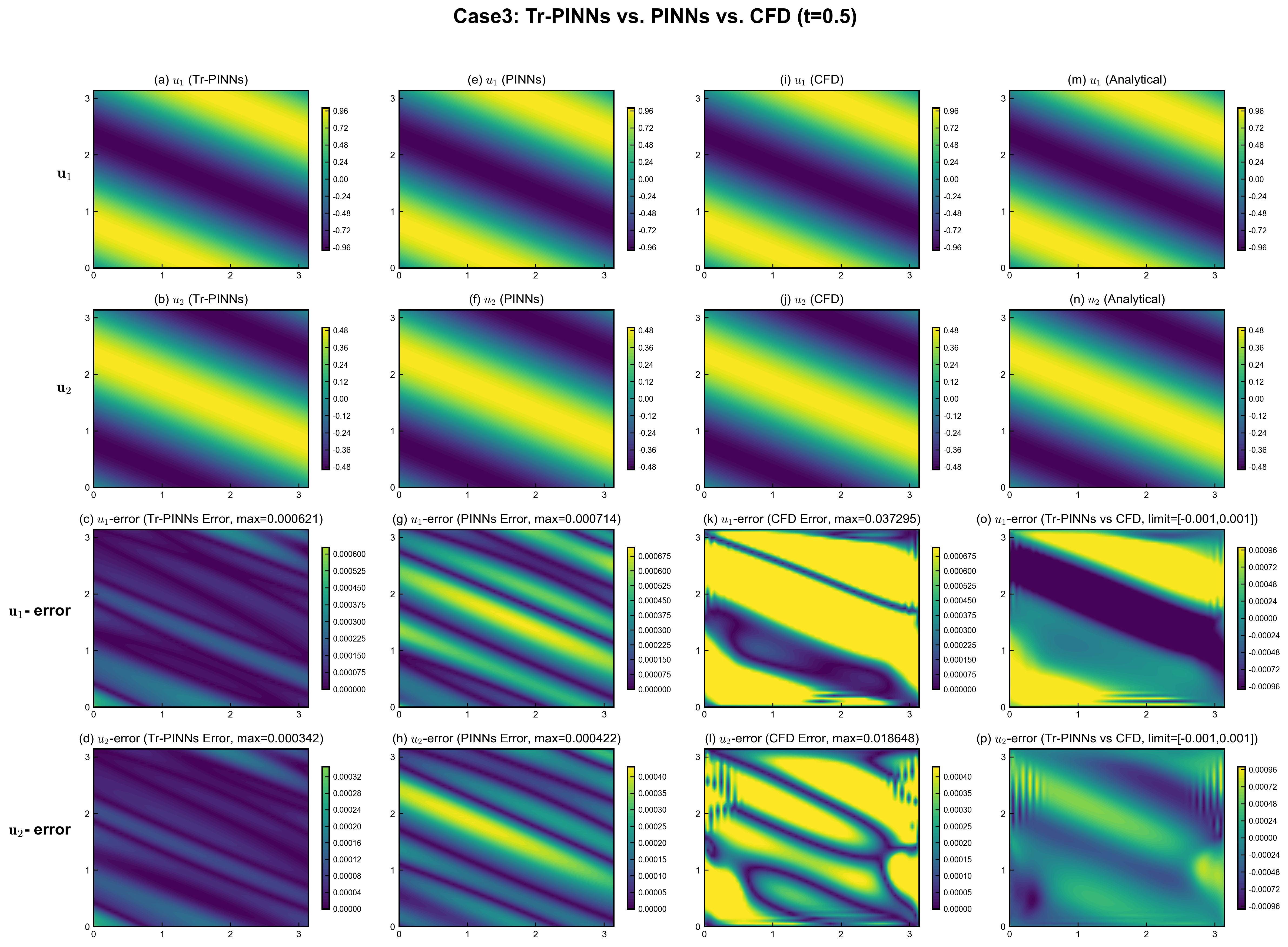}
  \caption{The third case: (a)(b) and (c)(d) show the predicted solutions and point-wise absolute errors of the Tr-PINNs algorithm; (e)(f) and (g)(h) show the predicted solutions and point-wise absolute errors of the PINNs algorithm; (i)(j) and (k)(l) show the numerical simulation results and point-wise absolute errors of OpenFOAM; (m)(n) denote the analytical solutions; (o)(p) show the point-wise errors between the Tr-PINNs predictions and OpenFOAM results.}
  \label{fig4.9}
\end{figure}

\section{Conclusion}
Compared with CFD methods, PINNs exhibit exceptionally high real-time efficiency in the inference stage phase, flexible mesh-free modeling and natural adaptability to multiphysics coupling. These merits effectively break through the core bottlenecks of traditional CFD in industrial engineering applications, including high computational latency, expensive modeling costs and insufficient adaptability to complex working conditions.

The Tr-PINNs algorithm proposed in this study achieves a remarkable improvement in boundary regularity through rigorous mathematical derivation and verification. The optimization mechanism embedded in this model further enhances its ability to capture key characteristics of complex physical fields. Particularly accurate simulation is realized for boundary regions, and excellent predictive performance is demonstrated in critical regions with high flow field gradients and peak values. This research provides a novel feasible idea and technical pathway for the development of numerical simulation and engineering applications in computational fluid dynamics.

\bigskip
{\bf Acknowledgments}
This work is supported by the Science and Technology Project of Beijing Municipal
 Education Commission, China (KM202210005011).

{\bf Conflict of interests}
This work does not have any conflicts of interest.

\begin{appendices}
  \numberwithin{equation}{section} 
  \titleformat{\section}
    {\normalfont\Large\bfseries}  
    {Appendix \thesection}        
    {0em}                        
    {}                            
    []                            

  \section{}  
\begin{Lemma} \label{lemma3}
Let $\Omega \subseteq \mathbb{R}^n$ ($n \geq 2$) be a bounded domain $\Omega$ with locally lipschitzian or $C^2$ boundary $\partial\Omega$. For all $(f,g) \in H^{-1}(\Omega) \times V^{\frac{1}{2}}(\partial\Omega)$, the following  nonhomogeneous Stokes 
problem
\begin{equation*}
  -\nu\Delta w + \nabla \pi =f \quad and \quad \nabla \cdot  w=b \quad in \, \Omega,\quad w=g\quad on\, \partial \Omega,
\end{equation*}
admits a unique solution $(w,\pi) \in H^1(\Omega) \times L^2(\Omega) $. Moreover the following estimate holds:\\
\begin{equation} 
\|w\|_{H^1(\Omega)}+\|\pi\|_{L^2(\Omega)}\leq C(\|f\|_{H^{-1}(\Omega)}+\|g\|_{H^{\frac{1}{2}}(\partial \Omega)} + \|b\|_{L^2(\Omega)}),
\end{equation}
where $C=C(n,\Omega)$.
\end{Lemma}
This result is stated in [\cite{G1998}, Theorem 1.1 and Exercise 1.1, Chapter 4].
\par\vspace{8pt}
 $\mathbf{Proof~of~Theorem~\ref{Unique1.1}:}$  
\setcounter{equation}{0} 
  \renewcommand{\theequation}{\arabic{equation}}
Inspired by \cite{R2007}, for equation \eqref{eq1.1}, we seek a solution $u$ of the form $u=w+v$,
  where $w$ is  a weak solution of the following equation
\begin{equation}\label{divequ1}
  -\nu\Delta w(t) + \nabla \pi(t) =0 \quad and \quad \nabla \cdot w(t)=0 \quad in \, \Omega,\quad w(t)=g(t)\quad on\, \partial \Omega,
\end{equation} 
for all $t \in [0,T]$, and $v$ is a solution to 
\begin{eqnarray} \label{equation3.2}
\begin{aligned}
  & \left\{
  \begin{array}{ll}
  \frac{\partial v}{\partial t}- \nu \Delta v + w \cdot \nabla v + v \cdot \nabla w+ v\cdot \nabla v + \nabla p = F,  &x\in\Omega,\quad t>0,  \\
     \nabla \cdot v=0,   &x\in\Omega,\quad t>0,\\
    v |_{t=0}=u_0 - w(0),   &x\in\Omega, \\
     v  =0,   &x\in\partial\Omega, ~~ t>0,\\
      \end{array}
  \right.
  \end{aligned}
 \end{eqnarray}   
where $F = f - \frac{\partial w}{\partial t} -  w \cdot \nabla w$.\\
From Lemma \ref{lemma3}, we know that when $g \in H^1(0,T;H^{\frac{1}{2}}(\partial \Omega))$, the equation \eqref{divequ1}  exists a unique solution $w \in H^1(0,T;H^1(\Omega))$,  thus $\frac{\partial w}{\partial t} \in L^2(0,T;H^1(\Omega))$ and $v_0=u_0 - w(0)\in H$.
Furthermore,  Simple computation shows that $F \in L^2(0,T;V')$,  then one can construct the approximate solutions of \eqref{equation3.2} by the classical
Faedo-Galerkin method and obtained the unique weak solution of equation \eqref{equation3.2} by compactness argument \cite{L1969}.
\hfill $\square$

\begin{Lemma}\label{lemma6}
Suppose \(\mathcal{X} \subset \mathbb{R}^n\) (for \(n \geq 1\)) is a measurable domain.  Let \(\mu\) be a probability measure on \(\mathcal{X}\). For all \(f \in L^2(\mathcal{X}, \mu)\), 
let \(\{\mathbf{x}_1, \mathbf{x}_2, \dots, \mathbf{x}_M\} \subset \mathcal{X}\) be independent and identically distributed samples drawn from \(\mu\). Define the Monte Carlo approximation of \(I = \int_{\mathcal{X}} f(\mathbf{x}) d\mu(\mathbf{x})\) as:
\begin{equation*}
  \hat{I}_M = \frac{1}{M} \sum_{i=1}^M f(\mathbf{x}_i).
\end{equation*}
This approximation admits the following two properties.\\
1) Unbiasedness:
\begin{equation}
 \mathbb{E}[\hat{I}_M] = I, \notag
\end{equation}
 where \(\mathbb{E}[\cdot]\) denotes the expectation with respect to the measure \(\mu\).\\
2) For any confidence level \(\delta \in (0,1)\), the inequality holds
\begin{equation} \notag
 \mathbb{P}\left( |\hat{I}_M - I| \leq \sigma \sqrt{\frac{\log(2/\delta)}{2M}} \right) \geq 1 - \delta  ,
\end{equation}
where \(\sigma^2 = \text{Var}(f) = \int_{\mathcal{X}} (f(\mathbf{x}) - I)^2 d\mu(\mathbf{x})\) is the finite variance of f.
The bound in the error estimate depends only on the variance \(\sigma^2\) of f, the sample size M and the confidence level \(\delta\).
\end{Lemma}
This is stated in [\cite{BZ2020}, Chapter 1].
\hfill $\square$
\end{appendices}

\end{document}